\newif\ifprint
\printfalse
\documentclass[twoside,10pt]{HandbookOfModuli}

\usepackage{amsmath,amsthm}
\usepackage{amssymb}
\usepackage{latexsym}
\usepackage[utf8]{inputenc}
\usepackage{textcomp}
\usepackage{color}
\usepackage{graphicx}
\usepackage{titlesec}
\usepackage{xspace}

\ifprint
	\input{giovanni} 
	\usepackage[final]{microtype}
\fi
\usepackage{bm}
\renewcommand{\mathbf}[1]{\bm{#1}} 

\ifprint
	\definecolor{linkred}{rgb}{0,0,0} 
	\definecolor{linkblue}{rgb}{0,0,0} 
\else
	\definecolor{linkred}{rgb}{0.7,0.2,0.2}
	\definecolor{linkblue}{rgb}{0,0.2,0.6}
\fi

\numberwithin{equation}{section} 

\usepackage[
	hypertexnames=false,
	hyperindex,
	pagebackref,
	pdftex,
	pdftitle={Handbook of Moduli},
	pdfdisplaydoctitle,
	pdfpagemode=UseNone,
	breaklinks=true,
	extension=pdf,
	bookmarks=false,
	plainpages=false,
	colorlinks,
	linkcolor=linkblue,
	citecolor=linkred,
	urlcolor=linkred,
	pdfmenubar=true,
	pdftoolbar=true,
	pdfpagelabels,
	pdfpagelayout=TwoPage,
	pdfview=Fit,
	pdfstartview=Fit
]{hyperref}

\catcode`@=11 \baselineskip 4.5mm
\def\ps@handbook{\def\@oddhead{\hfill \leftmark \hfill\thepage }
\def\@evenhead{\thepage \hfill \rightmark \hfill}
\def\@oddfoot{}
\def\@evenfoot{}}
\def\@evenhead{}
\def\@oddfoot{}
\def\@evenfoot{\hfill\copyright\ China Higher Education Press}
\def\list#1#2{\ifnum \@listdepth >5\relax \@toodeep \else \global
\advance \@listdepth\@ne \fi \rightmargin \z@ \listparindent\z@
\itemindent\z@ \csname @list\romannumeral\the\@listdepth\endcsname
\def\@itemlabel{#1}\let\makelabel\@mklab \@nmbrlistfalse #2\relax
\@trivlist \parskip -\parsep \parindent\listparindent \advance
\linewidth -\rightmargin \advance\linewidth -\leftmargin \advance
\@totalleftmargin \leftmargin \parshape \@ne \@totalleftmargin
\linewidth \ignorespaces}
\renewcommand*\l@section{\@tocline{1}{0pt}{0em}{1em}{}}
\renewcommand*\l@subsection{\@tocline{2}{0pt}{1.5em}{2em}{}} 
\renewcommand\contentsnamefont{\bfseries\large} 
\catcode`@=12
\renewcommand{\theequation}{\thesection.\arabic{equation}}

\def\thebibliography#1{\section*{References}
\list{[\arabic{enumi}]}{\settowidth \labelwidth{[#1]} \leftmargin
\labelwidth \advance \leftmargin \labelsep \usecounter{enumi}}
\def\newblock{\hskip .11em plus .33em minus .07em} \sloppy
\clubpenalty 4000 \widowpenalty 4000 \sfcode`\.=1000 \relax}

\titleformat{\section}{\normalfont\large\bfseries}{\thesection.}{0.5em}{}[\kern0.em]
\titleformat{\subsection}{\normalfont\bfseries}{\thesubsection.}{0.3em}{}[\kern0.em]
\titleformat{\subsubsection}[runin]{\normalfont\bfseries}{\thesubsubsection.}{0.5em}{}[\kern0.5em]

\def\fofsubsubsection#1{\refstepcounter{equation}\subsubsection*{\theequation.\kern0.25em #1}}
\def\foisubsubsection#1{\refstepcounter{equation}\subsubsection*{\kern\parindent\theequation.\kern0.25em #1}}

\usepackage[arrow,curve,matrix]{xy}

\newcounter{saveenumi}

\newtheorem{claim}[equation]{Claim}
\newtheorem{conj}[equation]{Conjecture}
\newtheorem{cor}[equation]{Corollary}
\newtheorem{defn}[equation]{Definition}
\newtheorem{notation}[equation]{Notation}
\newtheorem{notation-ass}[equation]{Notation-Assumption}
\newtheorem{myprop}[equation]{Proposition}
\newtheorem{mythm}[equation]{Theorem}

\theoremstyle{remark}
\newtheorem{cons}[equation]{Consequence}
\newtheorem{construction}[equation]{Construction}
\newtheorem{example}[equation]{Example}
\newtheorem{fact}[equation]{Fact}

\newtheorem{obs}[equation]{Observation}
\newtheorem{rem}[equation]{Remark}
\newtheorem{questions}[equation]{Questions}
\newtheorem{question}[equation]{Question}
\newtheorem{sass}[equation]{Simplifying Assumptions}
\newtheorem{warning}[equation]{Warning}

\DeclareMathOperator{\sA}{\mathcal A}
\DeclareMathOperator{\sO}{\mathcal O}

\DeclareMathOperator{\codim}{codim}
\DeclareMathOperator{\Id}{Id}
\DeclareMathOperator{\Image}{Image}

\DeclareMathOperator{\Pic}{Pic}
\DeclareMathOperator{\Spec}{Spec}
\DeclareMathOperator{\supp}{supp}
\DeclareMathOperator{\Sym}{Sym}
\DeclareMathOperator{\Var}{Var}

\begin{document}
\setcounter{page}{1}
%
%
\long\def\replace#1{#1}

%
%
\title[Differential forms, MMP, and Hyperbolicity]{Differential forms on
  singular spaces, the minimal model program, and hyperbolicity of moduli
  stacks}

\date{\today}

\dedicatory{In memory of Eckart Viehweg}

%
%
\author{Stefan Kebekus}
\address{Stefan Kebekus, Albert-Ludwigs Universität Freiburg, Mathematisches Institut, Eckerstraße 1, 79104 Freiburg im Breisgau, Germany}
\email{stefan.kebekus@math.uni-freiburg.de}

%
%

\begin{abstract}
  The Shafarevich Hyperbolicity Conjecture, proven by Arakelov and Parshin,
  considers a smooth, projective family of algebraic curves over a smooth
  quasi-projective base curve $Y$. It asserts that if $Y$ is of special type,
  then the family is necessarily isotrivial.

  This survey discusses hyperbolicity properties of moduli stacks and
  generalisations of the Shafarevich Hyperbolicity Conjecture to higher
  dimensions. It concentrates on methods and results that relate moduli theory
  with recent progress in higher dimensional birational geometry.
\end{abstract}  

\maketitle
\thispagestyle{empty}
%
%

\tableofcontents

\section{Introduction}

\subsection{The Shafarevich hyperbolicity conjecture}

\subsubsection{Statement}

In his contribution to the 1962 International Congress of Mathematicians, Igor
Shafarevich formulated an influential conjecture, considering smooth,
projective families $f^\circ: X^\circ \to Y^\circ$ of curves of genus $g > 1$,
over a fixed smooth quasi-projective base curve $Y^\circ$.  One part of the
conjecture, known as the ``hyperbolicity conjecture'', gives a sufficient
criterion to guarantee that any such family is isotrivial. The conjecture was
shown in two seminal works by Parshin and Arakelov, including the following
special case.

\begin{mythm}[Shavarevich Hyperbolicity Conjecture, \cite{Shaf63}, \cite{Parshin68, Arakelov71}]\label{thm:shafarevich}
  Let $f^\circ: X^\circ \to Y^\circ$ be a smooth, complex, projective family
  of curves of genus $g > 1$, over a smooth quasi-projective base curve
  $Y^\circ$.  If $Y^\circ$ is isomorphic to one of the following varieties,
  \begin{itemize}
  \item the projective line $\mathbb P^1$,
  \item the affine line $\mathbb A^1$,
  \item the affine line minus one point $\mathbb C^*$, or
  \item an elliptic curve,
  \end{itemize}
  then any two fibres of $f^\circ$ are necessarily isomorphic.
\end{mythm}

\begin{notation-ass}
  Throughout this paper, a \emph{family} is a flat morphism of algebraic
  varieties with connected fibres. We always work over the complex number
  field.
\end{notation-ass}

\begin{rem}
  Following standard convention, we refer to Theorem~\ref{thm:shafarevich} as
  ``Shafarevich hyperbolicity conjecture'' rather than ``Arakelov-Parshin
  theorem''.  The reader interested in a complete picture is referred to
  \cite[p.~253ff]{Viehweg01}, where all parts of the Shafarevich conjecture
  are discussed in more detail.
\end{rem}

Formulated in modern terms, Theorem~\ref{thm:shafarevich} asserts that any
morphism from a smooth, quasi-projective curve $Y^\circ$ to the moduli stack
of algebraic curves is necessarily constant if $Y^\circ$ is one of the special
curves mentioned in the theorem. If $Y^\circ$ is a quasi-projective variety of
arbitrary dimension, then any morphism from $Y^\circ$ to the moduli stack
contracts all rational and elliptic curves, as well as all affine lines and
$\mathbb C^*$s that are contained in $Y^\circ$.

We refer to \cite[Sect.~16.E.1]{HK10} for a discussion of the relation between
the Shafarevich hyperbolicity conjecture and the notions of Brody-- and Kobayashi
hyperbolicity.

\subsubsection{Aim and scope}

This survey is concerned with generalisations of the Shafarevich hyperbolicity
conjecture to higher dimensions, concentrating on methods and results that
relate moduli-- and minimal model theory. We hope that the methods presented
here will be applicable to a much wider ranges of problems, in moduli theory
and elsewhere.  The list of problems that we would like to address include the
following.

\begin{questions}
  Apart from the quasi-projective curves mentioned above, what other varieties
  admit only constant maps to the moduli stack of curves? What about moduli
  stacks of higher dimensional varieties? Given a variety $Y^\circ$, is there
  a good geometric description of the subvarieties that will always be
  contracted by any morphism to any reasonable moduli stack?
\end{questions}

Much progress has been achieved in the last years and several of the questions
can be answered today. It turns out that there is a close connection between
the minimal model program of a given quasi-projective variety $Y^\circ$, and
its possible morphisms to moduli stacks. Some of the answers obtained are in
fact consequences of this connection.

In the limited number of pages available, we say almost nothing about the
history of higher dimensional moduli, or about the large body of important
works that approach the problem from other points of view. Hardly any serious
attempt is made to give a comprehensive list of references, and the author
apologises to all those whose works are not adequately represented here, be it
because of the author's ignorance or simply because of lack of space.

The reader who is interested in a broader overview of higher dimensional
moduli, its history, complete references, and perhaps also in rigidity
questions for morphisms to moduli stacks is strongly encouraged to consult the
excellent surveys found in this handbook and elsewhere, including \cite{HK10,
  Kovacs06a, Viehweg01}. A gentle and very readable introduction to moduli of
higher dimensional varieties in also found in \cite{MR2483953}, while
Viehweg's book \cite{V95} serves as a standard technical reference for the
construction of moduli spaces.

Most relevant notions and facts from minimal model theory can either be found
in the introductory text \cite{Matsuki02}, or in the extremely clear and
well-written reference book \cite{KM98}. Recent progress in minimal model
theory is surveyed in \cite{HK10}.

\subsection{Outline of this paper}
\label{ssec:outline}

Section~\ref{sec:generalization} introduces a number of conjectural
generalisations of the Shafarevich hyperbolicity conjecture and gives an
overview of the results that have been obtained in this direction. In
particular, we mention results relating the moduli map and the minimal model
program of the base of a family.

Sections~\ref{sec:VZ} and \ref{sec:ext} introduce the reader to methods that
have been developed to attack the conjectures mentioned in
Section~\ref{sec:generalization}. While Section~\ref{sec:VZ} concentrates on
positivity results on moduli spaces and on Viehweg and Zuo's construction of
(pluri-)differential forms on base manifolds of families,
Section~\ref{sec:ext} summarises results concerning differential forms on
singular spaces. Both sections contain sketches of proofs which aim to give an
idea of the methods that go into the proofs, and which might serve as a
guideline to the original literature. The introduction to
Section~\ref{sec:ext} motivates the results on differential forms by
explaining a first strategy of proof for a special case of a (conjectural)
generalisation of the Shafarevich hyperbolicity conjecture. Following this
plan of attack, a more general case is treated in the concluding
Section~\ref{sec:ideaOfProof}, illustrating the use of the methods introduced
before.

\subsection{Acknowledgements}

This paper is dedicated to the memory of Eckart Viehweg. Like so many other
mathematicians of his age group, the author benefited immensely from Eckart's
presence in the field, his enthusiasm, guidance and support. Eckart will be
remembered as an outstanding mathematician, and as a fine human being.

\medskip

The work on this paper was partially supported by the DFG Forschergruppe 790
``Classification of algebraic surfaces and compact complex
manifolds''. Patrick Graf kindly read earlier versions of this paper and
helped to remove several errors and misprints. Many of the results presented
here have been obtained in joint work of Sándor Kovács and the author. The
author would like to thank Sándor for innumerable discussions, and for a long
lasting collaboration. He would also like to thank the anonymous referee for
careful reading and for numerous suggestions that helped to improve the
quality of this survey.

Not all the material presented here is new, and some parts of this survey have
appeared in similar form elsewhere. The author would like to thank his
coauthors for allowing him to use material from their joint research
papers. The first subsection of every chapter lists the sources that the
author was aware of.

\section{Generalisations of the Shafarevich hyperbolicity conjecture}
\label{sec:generalization}

\subsection{Families of higher dimensional varieties}

Given its importance in algebraic and arithmetic geometry, much work has been
invested to generalise the Shafarevich hyperbolicity conjecture,
Theorem~\ref{thm:shafarevich}. Historically, the first generalisations have
been concerned with families $f^\circ : X^\circ \to Y^\circ$ where $Y^\circ$
is still a quasi-projective curve, but where the fibres of $f^\circ$ are
allowed to have higher dimension. The following elementary example shows,
however, that Theorem~\ref{thm:shafarevich} cannot be generalised na\"ively,
and that additional conditions must be posed.

\begin{example}[Counterexample to the Shafarevich hyperbolicity conjecture for higher dimensional fibers]\label{ex:counter}
  Consider a smooth projective surface $Y$ of general type which contains a
  rational or elliptic curve $C \subset Y$. Assume that the automorphism group
  of $Y$ fixes the curve $C$ pointwise. Examples can be obtained by choosing
  any surface of general type and then blowing up sufficiently many points in
  sufficiently general position ---each blow-up will create a rational curve
  and lower the number of automorphisms. Thus, if $c_1$ and $c_2 \in C$ are
  any two distinct points, then the surfaces $Y_{c_i}$ obtained by blowing up
  the points $c_i$ are non-isomorphic.

  In order to construct a proper family, consider the product $Y \times C$
  with its projection $\pi: Y \times C \to C$ and with the natural section
  $\Delta \subset Y \times C$. If $X$ is the blow-up of $Y \times C$ in
  $\Delta$, then we obtain a smooth, projective family $f: X \to C$ of
  surfaces of general type, with the property that no two fibres are
  isomorphic.
\end{example}

It can well be argued that Counterexample~\ref{ex:counter} is not very
natural, and that the fibres of the family $f$ would trivially be isomorphic
if they had not been blown up artificially. This might suggest to consider
only families that are ``not the blow-up of something else''. One way to make
this condition is precise is to consider only \emph{families of minimal
  surfaces}, i.e., surfaces $F$ whose canonical bundle $K_F$ is semiample. In
higher dimensions, it is often advantageous to impose a stronger condition and
consider only \emph{families of canonically polarised manifolds}, i.e.,
manifolds $F$ whose canonical bundle $K_F$ is ample.

Hyperbolicity properties of families of minimal surfaces and families of
minimal varieties have been studied by a large number of people, including
Migliorini \cite{Migliorini95}, Kovács \cite{Kovacs96e, Kovacs97a} and
Oguiso-Viehweg \cite{Oguiso-Viehweg01}. For families of canonically polarised
manifolds, the analogue of Theorem~\ref{thm:shafarevich} has been shown by
Kovács in the algebraic setup \cite{Kovacs00a}. Combining algebraic arguments
with deep analytic methods, Viehweg and Zuo prove a more general Brody
hyperbolicity theorem for moduli spaces of canonically polarised manifolds
which also implies an analogue of Theorem~\ref{thm:shafarevich},
\cite{Vie-Zuo03a}.

\begin{mythm}[Hyperbolicity for families of canononically polarized varieties, \cite{Kovacs00a, Vie-Zuo03a}]\label{thm:kovhyp}
  Let $f^\circ: X^\circ \to Y^\circ$ be a smooth, complex, projective family
  of canonically polarised varieties of arbitrary dimension, over a smooth
  quasi-projective base curve $Y^\circ$. Then the conclusion of the
  Shafarevich hyperbolicity conjecture, Theorem~\ref{thm:shafarevich},
  holds. \qed
\end{mythm}

\subsection{Families over higher dimensional base manifolds}

This paper discusses generalisations of the Shafarevich hyperbolicity
conjecture to families over higher dimensional base manifolds. To formulate
any generalisation, two points need to be clarified.
\begin{enumerate}
\item We need to define a higher dimensional analogue for the list of
  quasi-projective curves given in Theorem~\ref{thm:shafarevich}.

\item Given any family $f^\circ : X^\circ \to Y^\circ$ over a
  higher dimensional base, call two points $y_1$, $y_2 \in Y^\circ$ equivalent
  if the fibres $(f^\circ)^{-1}(y_1)$ and $(f^\circ)^{-1}(y_2)$ are
  isomorphic. If $Y^\circ$ is a curve, then either there is only one
  equivalence class, or all equivalence classes are finite. For families over
  higher dimensional base manifolds, the equivalence classes will generally be
  subvarieties of arbitrary dimension. We will need to have a quantitative
  measure for the number of equivalence classes and their dimensions.
\end{enumerate}

The problems outlined above justify the definition of the \emph{logarithmic
  Kodaira dimension} and of the \emph{variation of a family},
respectively. Before coming to the generalisations of the Shafarevich
hyperbolicity conjecture in Section~\ref{ssec:ViehwegConj} below, we recall
the definitions for the reader's convenience.

\subsubsection{The logarithmic Kodaira dimension}

The logarithmic Kodaira dimension generalises the classical notion of Kodaira
dimension to the category of quasi-projective varieties. 

\begin{defn}[Logarithmic Kodaira dimension]
  Let $Y^\circ$ be a smooth quasi-projective variety and $Y$ a smooth
  projective compactification of $Y^\circ$ such that $D := Y \setminus
  Y^\circ$ is a divisor with simple normal crossings. The logarithmic Kodaira
  dimension of $Y^\circ$, denoted by $\kappa(Y^\circ)$, is defined to be the
  Kodaira-Iitaka dimension of the line bundle $\sO_Y(K_Y + D) \in \Pic(Y)$.  A
  quasi-projective variety $Y^\circ$ is called of \emph{log general type} if
  $\kappa(Y^\circ)=\dim Y^\circ$, i.e., the divisor $K_Y+D$ is big.
\end{defn}

It is a standard fact of logarithmic geometry that a compactification $Y$ with
the described properties exists, and that the logarithmic Kodaira dimension
$\kappa(Y^\circ)$ does not depend on the choice of the compactification.

\begin{obs}
  The quasi-projective curves listed in Theorem~\ref{thm:shafarevich} are
  precisely those curves $Y^\circ$ with logarithmic Kodaira dimension
  $\kappa(Y^\circ) \leq 0$.
\end{obs}

\subsubsection{The variation of a family}

The following definition provides a quantitative measure of the
\emph{birational} variation of a family. Note that the definition is
meaningful even in cases where no moduli space exists.

\begin{defn}[\protect{Variation of a family, cf.~\cite[Introduction]{Viehweg83}}]\label{def:var}
  Let $f^\circ: X^\circ \to Y^\circ$ be a projective family over an
  irreducible base $Y^\circ$, and let $\overline{\mathbb C(Y^\circ)}$ denote
  the algebraic closure of the function field of $Y^\circ$. The variation of
  $f^\circ$, denoted by $\Var(f^\circ)$, is defined as the smallest integer
  $\nu$ for which there exists a subfield $K$ of $\overline{\mathbb
    C(Y^\circ)}$, finitely generated of transcendence degree $\nu$ over
  $\mathbb C$ and a $K$-variety $F$ such that $X\times_{Y^\circ} \Spec
  \overline{\mathbb C(Y^\circ)}$ is birationally equivalent to $F\times_{\Spec
    K}\Spec\overline{\mathbb C(Y^\circ)}$.
\end{defn}

\begin{rem}
  In the setup of Definition~\ref{def:var}, assume that all fibres if
  $Y^\circ$ are canonically polarised complex manifolds. Then coarse moduli
  schemes are known to exist, \cite[Thm.~1.11]{V95}, and the variation is the
  same as either the dimension of the image of $Y^\circ$ in moduli, or the
  rank of the Kodaira-Spencer map at the general point of $Y^\circ$.  Further,
  one obtains that $\Var(f^\circ) = 0$ if and only if all fibres of $f^\circ$
  are isomorphic. In this case, the family $f^\circ$ is called ``isotrivial''.
\end{rem}

\subsubsection{Viehweg's conjecture}
\label{ssec:ViehwegConj}

Using the notion of ``logarithmic Kodaira dimension'' and ``variation'', the
Shafarevich hyperbolicity conjecture can be reformulated as follows.
\begin{mythm}[Reformulation of Theorem~\ref{thm:shafarevich}]
  If $f^\circ: X^\circ \to Y^\circ$ is any smooth, complex, projective family
  of curves of genus $g > 1$, over a smooth quasi-projective base curve
  $Y^\circ$, and if $\Var(f^\circ) = \dim Y^\circ$, then $\kappa(Y^\circ) =
  \dim Y^\circ$.
\end{mythm}

Aiming to generalise the Shafarevich hyperbolicity conjecture to families over
higher dimensional base manifolds, Viehweg has conjectured that this
reformulation holds true in arbitrary dimension.

\begin{conj}[\protect{Viehweg's conjecture, \cite[6.3]{Viehweg01}}]\label{conj:Viehweg}
  Let $f^\circ: X^\circ \to Y^\circ$ be a smooth projective family of
  varieties with semiample canonical bundle, over a quasi-projective manifold
  $Y^\circ$. If $f^\circ$ has maximal variation, then $Y^\circ$ is of log
  general type. In other words,
  $$
  \Var(f^\circ) = \dim Y^\circ \, \Rightarrow \, \kappa(Y^\circ) = \dim
  Y^\circ.
  $$
\end{conj}

Viehweg's conjecture has been proven by Sándor Kovács and the author in case
where $Y^\circ$ is a surface, \cite{KK08, KK07b}, or a threefold,
\cite{KK08c}.  The methods developed in these papers will be discussed, and an
idea of the proof will be given later in this paper, cf.~the outline of this
paper given in Section~\ref{ssec:outline} on page~\pageref{ssec:outline}.

\begin{mythm}[\protect{Viehweg's conjecture for families over threefolds, \cite[Thm.~1.1]{KK08c}}]\label{thm:mainresult0}
  Viehweg's conjecture holds in case where $\dim Y^\circ \leq 3$. \qed
\end{mythm}

For families of \emph{canonically polarised} varieties, much stronger results
have been obtained, giving an explicit geometric explanation of
Theorem~\ref{thm:mainresult0}.

\begin{mythm}[\protect{Relationship between the moduli map and the MMP, \cite[Thm.~1.1]{KK08c}}]\label{thm:mainresult2}
  Let $f^\circ: X^\circ \to Y^\circ$ be a smooth projective family of
  canonically polarised varieties, over a quasi-projective manifold $Y^\circ$
  of dimension $\dim Y^\circ \leq 3$. Let $Y$ be a smooth compactification of
  $Y^\circ$ such that $D := Y \setminus Y^\circ$ is a divisor with simple
  normal crossings.

  Then any run of the minimal model program of the pair $(Y,D)$ will terminate
  in a Kodaira or Mori fibre space whose fibration factors the moduli map
  birationally.
\end{mythm}

\begin{rem}
  Neither the compactification $Y$ nor the minimal model program discussed in
  Theorem~\ref{thm:mainresult2} is unique. When running the minimal model
  program, one often needs to choose the extremal ray that is to be
  contracted.
\end{rem}

In order to explain the statement of Theorem~\ref{thm:mainresult2}, let
$\mathfrak M$ be the appropriate coarse moduli space whose existence is shown,
e.g.~in \cite[Thm.~1.11]{V95}. Further, let $\mu^\circ : Y^\circ \to \mathfrak
M$ be the moduli map associated with the family $f^\circ$, and let $\mu : Y
\dasharrow \mathfrak M$ be the associated rational map from the
compactification $Y$. If $\lambda : Y \dasharrow Y_\lambda$ is a rational map
obtained by running the minimal model program, and if $Y_\lambda \to
Z_\lambda$ is the associated Kodaira or Mori fibre space, then
Theorem~\ref{thm:mainresult2} asserts the existence of a map $Z_\lambda
\dasharrow \mathfrak M$ that makes the following diagram commutative,
$$
\xymatrix{Y \ar@{-->}[rrr]^{\lambda}_{\text{MMP of the pair } (Y,D)}
  \ar@{-->}[d]_{\text{moduli map induced by } f^\circ} &&&
  Y_\lambda \ar[d]^{\text{Kodaira or Mori fibre
space}} \\
\mathfrak M &&& Z_\lambda \ar@{-->}[lll]^{\exists !}.}
$$

Now, if we assume in addition that $\kappa(Y^\circ) \geq 0$, then the minimal
model program terminates in a Kodaira fibre space whose base $Z_\lambda$ has
dimension $\dim Z_\lambda = \kappa(Y^\circ)$, so that $\Var(f^\circ) \leq
\kappa(Y^\circ)$. If we assume that $\kappa(Y^\circ) = - \infty$, then the
minimal model program terminates in proper Mori fibre space and we obtain that
$\dim Z_\lambda < \dim Y$ and $\Var(f^\circ) < \dim Y^\circ$. The following
refined answer to Viehweg's conjecture is therefore an immediate corollary of
Theorem~\ref{thm:mainresult2}.

\begin{cor}[\protect{Refined answer to Viehweg's conjecture, \cite[Cor.~1.3]{KK08c}}]\label{cor:mainresult1}
  Let $f^\circ: X^\circ \to Y^\circ$ be a smooth projective family of
  canonically polarised varieties, over a quasi-projective manifold $Y^\circ$
  of dimension $\dim Y^\circ \leq 3$. Then either
  \begin{enumerate}
  \item $\kappa(Y^\circ) = -\infty$ and $\Var(f^\circ) < \dim Y^\circ$, or
  \item $\kappa(Y^\circ) \geq 0$ and $\Var(f^\circ) \leq
    \kappa(Y^\circ)$. \qed
  \end{enumerate}
\end{cor}

\begin{rem}
  Corollary~\ref{cor:mainresult1} asserts that any family of canonically
  polarised varieties over a base manifold $Y^\circ$ with $\kappa(Y^\circ) =
  0$ is necessarily isotrivial.
\end{rem}

\begin{rem}
  Corollary~\ref{cor:mainresult1} has also been shown in case where $Y^\circ$ is
  a \emph{projective} manifold of arbitrary dimension, conditional to the
  standard conjectures of minimal model theory\footnote{i.e., existence and
    termination of the minimal model program and abundance},
  cf.~\cite[Thm.~1.4]{KK08b}. A very short proof that does not rely on minimal
  model theory has been announced by Patakfalvi as this paper goes to print,
  \cite{Pat11}.
\end{rem}

\begin{example}[Optimality of Corollary~\ref{cor:mainresult1} in case $\kappa(Y^\circ) = -\infty$]
  To see that the result of Corollary~\ref{cor:mainresult1} is optimal in case
  $\kappa(Y^\circ) = -\infty$, let $f_1^\circ : X_1^\circ \to Y^\circ_1$ be
  any family of canonically polarised varieties with $\Var(f_1^\circ) =2$,
  over a smooth surface $Y^\circ_1$ (which may or may not be compact). Setting
  $X^\circ := X_1^\circ \times \mathbb P^1$ and $Y^\circ := Y_1^\circ \times
  \mathbb P^1$, we obtain a family $f^\circ = f^\circ_1 \times {\rm
    Id}_{\mathbb P^1}: X^\circ \to Y^\circ$ with variation $\Var(f^\circ) =
  2$, and with a base manifold $Y^\circ$ of Kodaira dimension $\kappa(Y^\circ)
  = - \infty$.
\end{example}

\begin{example}[Related and complementary results in case $\kappa(Y^\circ) = -\infty$]
  In the setup of Corollary~\ref{cor:mainresult1}, if $Y^\circ$ is a projective
  Fano manifold, then a fundamental result of Campana and Kollár-Miyaoka-Mori
  asserts that $Y^\circ$ is rationally connected, \cite[V.~Thm.~2.13]{K96}. In
  other words, given any two points $x$, $y$ in $Y^\circ$, there exists a
  rational curve $C \subset Y^\circ$ which contains both $x$ and $y$. Recalling
  from Theorem~\ref{thm:kovhyp} that families over rational curves are
  isotrivial, it follows immediately that the family $f^\circ$ is necessarily
  isotrivial itself.

  A much stronger version of this result has been shown by Lohmann,
  \cite{Loh11}. Given a projective variety $Y$ and a $\mathbb Q$-divisor $D$
  such that $(Y,D)$ is a divisorially log terminal
  (=dlt)\footnote{\label{foot:dlt}We refer to \cite[Sect.~2.3]{KM98} for the
    definition of a \emph{dlt} pair, and for related notions concerning
    singularities of pairs that are relevant in minimal model theory.} pair,
  consider the smooth quasi-projective variety
  $$
  Y^\circ := (Y \setminus \supp \lfloor D \rfloor)_{\rm reg}.
  $$
  Lohmann shows that if $(Y,D)$ is log-Fano, that is, if the $\mathbb Q$-divisor
  $-(K_Y+D)$ is ample, then any family of canonically polarized varieties over
  $Y^\circ$ is necessarily isotrivial. The proof relies on a generalization of
  Araujo's result \cite{Ara} which relates extremal rays in the moving cone of a
  variety with fiber spaces that appear at the end of the minimal model
  program. Lohmann shows that the moduli map factorizes through any of the
  fibrations obtained in this way.
\end{example}

\subsubsection{Campana's conjecture}

In a series of papers, including \cite{Cam04, Cam07}, Campana introduced the
notion of ``geometric orbifolds'' and ``special varieties''. Campana's
language helps to formulate a very natural generalisation of
Theorem~\ref{thm:shafarevich}, which includes the cases covered by the Viehweg
Conjecture~\ref{conj:Viehweg}, and gives (at least conjecturally) a
satisfactory geometric explanation of isotriviality observed in some families
over spaces that are not covered by Conjecture~\ref{conj:Viehweg}.

Before formulating the conjecture, we briefly recall the precise definition of
a special logarithmic pair for the reader's convenience. We take the classical
Bogomolov-Sommese Vanishing Theorem as our starting point. We refer to
\cite{Iitaka82, EV92} or to the original reference \cite{Deligne70} for an
explanation of the sheaf $\Omega^p_Y(\log D)$ of logarithmic differentials.

\begin{mythm}[\protect{Bogomolov-Sommese Vanishing Theorem, cf.~\cite[Sect.~6]{EV92}}]\label{thm:classBSv}
  Let $Y$ be a smooth projective variety and $D \subset Y$ a reduced (possibly
  empty) divisor with simple normal crossings. If $p \leq \dim Y$ is any
  number and $\sA \subseteq \Omega^p_Y(\log D)$ any invertible subsheaf, then
  the Kodaira-Iitaka dimension of $\sA$ is at most $p$, i.e., $\kappa(\sA)
  \leq p$. \qed
\end{mythm}

In a nutshell, we say that a pair $(Y, D)$ is
special if the inequality in the Bogomolov-Sommese Vanishing Theorem is always
strict.

\begin{defn}[Special logarithmic pair]\label{def:speciallog}
  In the setup of Theorem~\ref{thm:classBSv}, a pair $(Y, D)$ is called
  \emph{special} if the strict inequality $\kappa(\sA) < p$ holds for all $p$
  and all invertible sheaves $\sA \subseteq \Omega^p_Y(\log D)$. A smooth,
  quasi-projective variety $Y^\circ$ is called special if there exists a
  smooth compactification $Y$ such that $D:= Y \setminus Y^\circ$ is a divisor
  with simple normal crossings and such that the pair $(Y,D)$ is special.
\end{defn}

\begin{rem}[Special quasi-projective variety]
  It is an elementary fact that if $Y^\circ$ is a smooth, quasi-projective
  variety and $Y_1$, $Y_2$ two smooth compactifications such that $D_i := Y_i
  \setminus Y^\circ$ are divisors with simple normal crossings, then $(Y_1,
  D_1)$ is special if and only if $(Y_2, D_2)$ is. The notion of special
  should thus be seen as a property of the quasi-projective variety $Y^\circ$.
\end{rem}

\begin{fact}[\protect{Examples of special manifolds, cf.~\cite[Thms.~3.22 and 5.1]{Cam04}}]
  Rationally connected manifolds and manifolds $X$ with $\kappa(X) = 0$ are
  special. \qed
\end{fact}

With this notation in place, Campana's conjecture can be formulated as
follows.

\begin{conj}[\protect{Campana's conjecture, \cite[Conj.~12.19]{Cam07}}]\label{conj:campana}
  Let $f: X^\circ \to Y^\circ$ be a smooth family of canonically polarised
  varieties over a smooth quasi-projective base. If $Y^\circ$ is special, then
  the family $f$ is isotrivial.
\end{conj}

In analogy with the construction of the maximally rationally connected
quotient map of uniruled varieties, Campana constructs in
\cite[Sect.~3]{Cam04} an almost-holomorphic ``core map'' whose fibres are
special in the sense of Definition~\ref{def:speciallog}. Like the MRC
quotient, the core map is uniquely characterised by certain maximality
properties, \cite[Thm.~3.3]{Cam04}, which essentially say that the core map of
$X$ contracts almost all special subvarieties contained in $X$. If Campana's
Conjecture~\ref{conj:campana} holds, this would imply that the core map always
factors the moduli map, similar to what we have seen in
Section~\ref{ssec:ViehwegConj} above,
$$
\xymatrix{ %
  Y \ar@{-->}[rrrr]^{\text{core map}}_{\text{almost holomorphic}}
  \ar@{-->}[d]_{\text{moduli map induced by } f^\circ} &&&&
  Z \ar@{-->}@/^0.4cm/[dllll]^{\exists !} \\
  \mathfrak M.}
$$

As with Viehweg's Conjecture~\ref{conj:Viehweg}, Campana's
Conjecture~\ref{conj:campana} has been shown for surfaces \cite{JK09} and
threefolds \cite{JKSpecialBaseManifolds}.

\begin{mythm}[\protect{Campana's conjecture in dimension three, \cite[Thm.~1.5]{JKSpecialBaseManifolds}}]\label{thm:main}
  Campana's Conjecture~\ref{conj:campana} holds if $\dim Y^\circ \leq 3$. \qed
\end{mythm}

\subsection{Conjectures and open problems}

Viehweg's Conjecture~\ref{conj:Viehweg} and Campana's
Conjecture~\ref{conj:campana} have been shown for families over base manifolds
of dimension three or less. As we will see in Section~\ref{sec:ideaOfProof},
the restriction to three-dimensional base manifolds comes from the fact that
minimal model theory is particularly well-developed for threefolds, and from
our limited ability to handle differential forms on singular spaces of higher
dimension. We do not believe that there is a fundamental reason that restricts
us to dimension three, and we do believe that the relationship between the
moduli map and the MMP found in Theorem~\ref{thm:mainresult2} will hold in
arbitrary dimension.

\begin{conj}[Relationship between the moduli map and the MMP]\label{conj:mainresult2}
  Let $f^\circ: X^\circ \to Y^\circ$ be a smooth projective family of
  canonically polarised varieties, over a quasi-projective manifold
  $Y^\circ$. Let $Y$ be a smooth compactification of $Y^\circ$ such that $D :=
  Y \setminus Y^\circ$ is a divisor with simple normal crossings. Then any
  run of the minimal model program of the pair $(Y,D)$ will terminate in a
  Kodaira or Mori fibre space whose fibration factors the moduli map
  birationally.
\end{conj}

\begin{conj}[\protect{Refined Viehweg conjecture, cf.~\cite[Conj.~1.6]{KK08}}]\label{conj:mainresult3}
  Corollary~\ref{cor:mainresult1} holds without the assumption that $\dim
  Y^\circ \leq 3$.
\end{conj}

Given the current progress in minimal model theory, a proof of
Conjectures~\ref{conj:mainresult2} and \ref{conj:mainresult3} does no longer
seem out of reach.

\section{Techniques I: Existence of Pluri-differentials on the base of a family}
\label{sec:VZ}

\subsection{The existence result}

Throughout the present Section~\ref{sec:VZ}, we consider a smooth projective
family $f^\circ: X^\circ \to Y^\circ$ of projective, canonically polarised
complex manifolds, over a smooth complex quasi-projective base. We assume that
the family is not isotrivial, and fix a smooth projective compactification $Y$
of $Y^\circ$ such that $D := Y \setminus Y^\circ$ is a divisor with simple
normal crossings. In this setup, Viehweg and Zuo have shown the following
fundamental result asserting the existence of many logarithmic
pluri-differentials on $Y$.

\begin{mythm}[Existence of pluri-differentials on $Y$, \protect{\cite[Thm.~1.4(i)]{VZ02}}]\label{thm:VZ}
  Let $f^\circ: X^\circ \to Y^\circ$ be a smooth projective family of
  canonically polarised complex manifolds, over a smooth complex
  quasi-projective base.  Assume that the family is not isotrivial and fix a
  smooth projective compactification $Y$ of $Y^\circ$ such that $D := Y
  \setminus Y^\circ$ is a divisor with simple normal crossings.

  Then there exists a number $m > 0$ and an invertible sheaf $\sA \subseteq
  \Sym^m \Omega^1_Y (\log D)$ whose Kodaira-Iitaka dimension is at least the
  variation of the family, $\kappa(\sA) \geq \Var(f^\circ)$. \qed
\end{mythm}

\begin{rem}
  Observe that the Shafarevich hyperbolicity conjecture,
  Theorem~\ref{thm:shafarevich}, follows as an immediate corollary of
  Theorem~\ref{thm:VZ}.
\end{rem}
\begin{rem}
  A somewhat weaker version of Theorem~\ref{thm:VZ} holds for families of
  projective manifolds with only semiample canonical bundle if one assumes
  additionally that the family is of maximal variation, i.e., that
  $\Var(f^\circ) = \dim Y^\circ$, cf.~\cite[Thm.~1.4(iv)]{VZ02}.
\end{rem}

As we will see in Section~\ref{sec:ideaOfProof}, the ``Viehweg-Zuo'' sheaf
$\sA$ is one of the crucial ingredients in the proofs of Viehweg's and
Campana's conjecture for families over threefolds,
Theorems~\ref{thm:mainresult0}, \ref{thm:mainresult2} and \ref{thm:main}.  A
careful review of Viehweg and Zuo's construction reveals that the
``Viehweg-Zuo sheaf'' $\sA$ comes from the coarse moduli space $\mathfrak{M}$,
at least generically. The precise statement, given in
Theorem~\ref{thm:VZimproved}, uses the following notion.

\begin{notation}[Differentials coming from moduli space generically]\label{not:introB}
  Let $\mu: Y^\circ \to \mathfrak M$ be the moduli map associated with the
  family $f^\circ$, and consider the subsheaf $\mathcal B \subseteq
  \Omega^1_Y(\log D)$, defined on presheaf level as follows: if $U \subseteq
  Y$ is any open set and $\sigma \in H^0\bigl(U,\, \Omega^1_Y (\log D) \bigr)$
  any section, then $\sigma \in H^0\bigl(U,\, \mathcal B \bigr)$ if and only
  if the restriction $\sigma|_{U'}$ is in the image of the differential map
  $$
  d\mu|_{U'} : \mu^* \bigl( \Omega^1_{\mathfrak M}\bigr)|_{U'} \longrightarrow
  \Omega^1_{U'},
  $$
  where $U' \subseteq U\cap Y^\circ$ is the open subset where the moduli map
  $\mu$ has maximal rank.
\end{notation}

\begin{rem}
  By construction, it is clear that the sheaf $\mathcal B$ is a saturated
  subsheaf of $\Omega^1_Y (\log D)$, i.e., that the quotient sheaf $\Omega^1_Y
  (\log D)/\mathcal B$ is torsion free. We say that $\mathcal B$ is the
  saturation of $\Image(d\mu)$ in $\Omega^1_Y(\log D)$.
\end{rem}

\begin{mythm}[\protect{Refinement of the Viehweg-Zuo Theorem~\ref{thm:VZ}, \cite[Thm.~1.4]{JK09}}]\label{thm:VZimproved}
  In the setup of Theorem~\ref{thm:VZ}, there exists a number $m > 0$ and an
  invertible subsheaf $\sA \subseteq \Sym^m \mathcal B$ whose Kodaira-Iitaka
  dimension is at least the variation of the family, $\kappa(\sA) \geq
  \Var(f^\circ)$.
\end{mythm}

Theorem~\ref{thm:VZimproved} follows without too much work from Viehweg's and
Zuo's original arguments and constructions, which are reviewed in
Section~\ref{ssec:VZ2} below. Compared with Theorem~\ref{thm:VZ}, the refined
Viehweg-Zuo theorem relates more directly to Campana's
Conjecture~\ref{conj:campana} and other generalizations of the Shafarevich
conjecture. To illustrate its use, we show in the surface case how
Theorem~\ref{thm:VZimproved} reduces Campana's Conjecture~\ref{conj:campana}
to the Viehweg Conjecture~\ref{conj:Viehweg}, for which a positive answer is
known.

\begin{cor}[Campana's conjecture in dimension two]
  Conjecture~\ref{conj:campana} holds if $\dim Y^\circ = 2$.
\end{cor}
\begin{proof}
  We maintain the notation of Conjecture~\ref{conj:campana} and let $f:
  X^\circ \to Y^\circ$ be a smooth family of canonically polarised varieties
  over a smooth quasi-projective base, with $Y^\circ$ a special surface.
  Since $Y^\circ$ is special, it is not of log general type, and the solution
  to Viehweg's conjecture in dimension two, \cite[Thm.~1.1]{KK08c}, gives that
  $\Var(f^\circ) < 2$.

  We argue by contradiction, suppose that $\Var(f^\circ)=1$ and choose a
  compactification $(Y, D)$ as in Definition~\ref{def:speciallog}. By
  Theorem~\ref{thm:VZimproved} there exists a number $m>0$ and an invertible
  subsheaf $\mathcal A \subseteq \Sym^m \mathcal B$ such that $\kappa(\mathcal
  A) \geq 1$.  However, since $\mathcal B$ is saturated in the locally free
  sheaf $\Omega_Y^1(\log D)$, it is reflexive, \cite[Claim on p.~158]{OSS},
  and since $\Var(f^\circ)=1$, the sheaf $\mathcal B$ is of rank $1$.  Thus
  $\mathcal B \subseteq \Omega_Y^1(\log D)$ is an invertible subsheaf,
  \cite[Lem.~1.1.15, on p.~154]{OSS}, and Definition~\ref{def:speciallog} of a
  special pair gives that $\kappa(\mathcal B) <1$, contradicting the fact that
  $\kappa(\mathcal A) \geq 1$.  It follows that $\Var(f^\circ) = 0$ and that
  the family is hence isotrivial.
\end{proof}

\subsubsection*{Outline of this section}

Given its importance in the theory, we give a very brief synopsis of
Viehweg-Zuo's proof of Theorem~\ref{thm:VZ}, showing how the theorem follows
from deep positivity results\footnote{The positivity results in question are
  formulated in Theorems~\ref{thm:Vpos} and Fact~\ref{fact:VZ48},
  respectively.} for push-forward sheaves of relative dualizing sheaves, and
for kernels of Kodaira-Spencer maps, respectively. Even though no proof of the
refined Theorem~\ref{thm:VZimproved}, is given, it is hoped that the reader
who chooses to read Section~\ref{ssec:VZ2} will believe that
Theorem~\ref{thm:VZimproved} follows with some extra work by essentially the
same methods.

The reader who is interested in a detailed understanding, including is
referred to the papers \cite{Kol86}, \cite{VZ02}, and to the survey
\cite{Viehweg01}. The overview contained in this section and the facts
outlined in Section~\ref{ssec:facts} can perhaps serve as a guideline to
\cite{VZ02}.

Many of the technical difficulties encountered in the full proof of
Theorem~\ref{thm:VZ} vanish if $f^\circ$ is a family of curves. The proof
becomes very transparent in this case. In particular, it is very easy to see
how the Kodaira-Spencer map associated with the family $f^\circ$ transports
the positivity found in push-forward sheaves into the sheaf of differentials
$\Omega^1_Y(\log Y)$. After setting up notation in Section~\ref{ssec:8A}, we
have therefore included a Section~\ref{ssec:8B} which discusses the curve case
in detail.

Most of the material presented in the current Section~\ref{sec:VZ}, including
the synopsis of Viehweg-Zuo's construction, is taken without much modification
from the paper \cite{JK09}. The presentation is inspired in part by
\cite{Viehweg01}.

\subsection{A synopsis of Viehweg-Zuo's construction}
\label{ssec:VZ2}

\subsubsection{Setup of notation}
\label{ssec:8A}

Throughout the present Section~\ref{ssec:VZ2}, we choose a smooth projective
compactification $X$ of $X^\circ$ such that the following holds:
\begin{enumerate}
\item The difference $\Delta := X \setminus X^\circ$ is a divisor with simple
  normal crossings.
\item The morphism $f^\circ$ extends to a projective morphism $f: X \to Y$.
  \setcounter{saveenumi}{\theenumi}
\end{enumerate}
It is then clear that $\Delta = f^{-1}(D)$ set-theoretically. Removing a
suitable subset $S \subset Y$ of codimension $\codim_Y S \geq 2$, the
following will then hold automatically on $Y' := Y\setminus S$ and $X' := X
\setminus f^{-1}(S)$, respectively.
\begin{enumerate}
\setcounter{enumi}{\thesaveenumi}
\item The restricted morphism $f' := f|_{X'}$ is flat.
\item The divisor $D' := D\cap Y'$ is smooth.
\item The divisor $\Delta' := \Delta \cap X'$ is a relative normal crossing
  divisor, i.e. a normal crossing divisor whose components and all their
  intersections are smooth over the components of $D'$.
\end{enumerate}
In the language of Viehweg-Zuo, \cite[Def~2.1(c)]{VZ02}, the restricted
morphism $f' : X' \to Y'$ is a ``good partial compactification of $f^\circ$''.

\begin{rem}[Restriction to a partial compactification]\label{rem:extoS}
  Let $\mathcal G$ be a locally free sheaf on $Y$, and let $\mathcal F'
  \subseteq \mathcal G|_{Y'}$ be an invertible subsheaf. Since $\codim_Y S
  \geq 2$, there exists a unique extension of the sheaf $\mathcal F'$ to an
  invertible subsheaf $\mathcal F \subseteq \mathcal G$ on $Y$. Furthermore,
  the restriction map $H^0 \bigl( Y,\, \mathcal F \bigr) \to H^0 \bigl( Y',\,
  \mathcal F' \bigr)$ is an isomorphism. In particular, the notion of
  Kodaira-Iitaka dimension makes sense for the sheaf $\mathcal F'$, and
  $\kappa(\mathcal F') = \kappa(\mathcal F)$.
\end{rem}

We denote the relative dimension of $X$ over $Y$ by $n := \dim X - \dim Y$.

\subsubsection{Idea of proof of Theorem~\ref*{thm:VZ} for families of curves}
\label{ssec:8B}

Before sketching the proof of Theorem~\ref{thm:VZ} in full generality, we
illustrate the main idea in a particularly simple setting.

\begin{sass}\label{sass:famcurves}
  Throughout the present introductory Section~\ref{ssec:8B}, we maintain the
  following simplifying assumptions in addition to the assumptions made in
  Theorem~\ref{thm:VZ}.
  \begin{enumerate}
  \item The quasi-projective variety $Y^\circ$ is in fact projective. In
    particular, we assume that $X=X^\circ$, $f = f^\circ$, that $D=\emptyset$
    and that $\Delta = \emptyset$.
  \item The family $f: X \to Y$ is a family of curves of genus $g > 1$. In
    particular, we have that $(T_{X/Y})^* = \Omega^1_{X/Y} = \omega_{X/Y}$,
    where $T_{X/Y}$ is the kernel of the derivative $Tf : T_X \to f^*(T_Y)$.
  \item The variation of $f^\circ$ is maximal, that is, $\Var(f^\circ) = \dim
    Y^\circ$.
  \end{enumerate}
\end{sass}

The proof of Theorem~\ref*{thm:VZ} sketched here uses positivity of the
push-forward of relative dualizing sheaves as its main input. The positivity
result required is discussed in Viehweg's survey \cite[Sect.~1--3]{Viehweg01},
where positivity is obtained as a consequence of generalised Kodaira vanishing
theorems. The reader interested in a broader overview might also want to look
at the remarks and references in \cite[Sect.~6.3.E]{L04}, as well as the
papers \cite{Kol86, Zuo00}

\begin{mythm}[\protect{Positivity of push-forward sheaves, cf.~\cite[Prop.~3.4.(i)]{VZ02}}]\label{thm:Vpos}
  Under the simplifying Assumptions~\ref{sass:famcurves}, the push-forward
  sheaf $f_* \bigl(\omega_{X/Y}^{\otimes 2} \bigr)$ is locally free of
  positive rank. If $\sA \in \Pic(Y)$ is any ample line bundle, then there
  exist numbers $N,M \gg 0$ and a sheaf morphism
  $$
  \phi : \sA^{\oplus M} \to \Sym^N f_* \bigl(\omega_{X/Y}^{\otimes 2} \bigr)
  $$
  which is surjective at the general point of $Y$. \qed
\end{mythm}

For the reader's convenience, we recall two other facts used in the proof,
namely the existence of a Kodaira-Spencer map, and Serre duality in the
relative setting.

\begin{mythm}[\protect{Kodaira-Spencer map,~cf.~\cite[Sect.~9.1.2]{Voisin-Hodge1} or \cite[Sect.~6.2]{Huy05}}]\label{thm:Kodaira-Spencer}
  Under the simplifying Assumptions~\ref{sass:famcurves}, since $\Var(f) > 0$,
  there exists a non-zero sheaf morphism $\kappa : T_Y \to R^1f_*(T_{X/Y})$
  which measures the variation of the isomorphism classes of fibres in
  moduli. \qed
\end{mythm}

\begin{mythm}[\protect{Serre duality in the relative setting, cf.~\cite[Sect.~6.4]{Liu02}}]\label{thm:serre}
  Under the simplifying Assumptions~\ref{sass:famcurves}, if $\mathcal F$ is
  any coherent sheaf on $X$, then there exists a natural isomorphism
  $f_*(\mathcal F^* \otimes \omega_{X/Y}) \cong \bigl(R^1 f_*(\mathcal
  F)\bigr)^* $. \qed
\end{mythm}
\begin{proof}[Proof of Theorem~\ref{thm:VZ} under the Simplifying Assumptions~\ref{sass:famcurves}]
  Consider the dual of the (non-trivial) Kodaira-Spencer map discussed in
  Theorem~\ref{thm:Kodaira-Spencer}, say $\kappa^* : \bigl(R^1f_*(T_{X/Y})
  \bigr)^* \to (T_Y)^*$. Recalling that $T_{X/Y}^*$ equals the relative
  dualizing sheaf $\omega_{X/Y}$, and using the relative Serre Duality
  Theorem~\ref{thm:serre}, the sheaf morphism $\kappa^*$ is naturally
  identified with a non-zero morphism
  \begin{equation}\label{eq:X2}
  \kappa^* : f_* \bigl( \omega_{X/Y}^{\otimes 2} \bigr) \to \Omega^1_Y.    
  \end{equation}
  Choosing an ample line bundle $\sA \in \Pic(Y)$ and sufficiently large and
  divisible numbers $N,M \gg 0$, Theorem~\ref{thm:Vpos} yields a sequence of
  sheaf morphisms
  $$
  \xymatrix{%
    \sA^{\oplus M} \ar[rr]^(.4){\phi}_(.4){\text{gen.~surjective}} && \Sym^N f_* \bigl(\omega_{X/Y}^{\otimes 2} \bigr)
    \ar[rr]^(.55){\Sym^N(\kappa^*)}_(.55){\text{non-trivial}} && \Sym^N \Omega^1_Y, %
  }
  $$
  whose composition $\sA^{\oplus M} \to \Sym^N \Omega^1_Y$ is clearly not the
  zero map. Consequently, we obtain a non-trivial map $\sA \to \Sym^N
  \Omega^1_Y$, finishing the proof of Theorem~\ref{thm:VZ} under the
  Simplifying Assumptions~\ref{sass:famcurves}.
\end{proof}

The proof outlined above uses the dual of the Kodaira-Spencer map as a vehicle
to transport the positivity which exists in $f_* \bigl( \omega_{X/Y}^{\otimes
  2} \bigr)$ into the sheaf $\Omega^1_Y$ of differential forms on $Y$. If $f$
was a family of surfaces rather than a family of curves, then Serre duality
cannot easily be used to identify the dual of $R^1f_*(T_{X/Y})$ with a
push-forward sheaf of type $f_* \bigl( \omega_{X/Y}^{\otimes \bullet} \bigr)$,
or any with other sheaf whose positivity is well-known.  To overcome this
problem, Viehweg suggested to replace the Kodaira-Spencer map $\kappa$ by
sequences of more complicated sheaf morphisms $\tau^0_{p,q}$ and $\tau^k$,
constructed in Sections~\ref{sssec:hKS1} and \ref{sssec:hKS2} below. To
motivate the slightly involved construction of these maps, we recall without
proof a description of the classical Kodaira-Spencer map.

\begin{construction}
  Under the Simplifying Assumptions~\ref{sass:famcurves}, consider the standard
  sequence of relative differential forms on $X$,
  \begin{equation}\label{eq:reldiff}
    0 \to f^*\Omega^1_Y \to \Omega^1_X \to \Omega^1_{X/Y} \to 0,    
  \end{equation}
  and its twist with the invertible sheaf $\omega_{X/Y}^*$,
  $$
  0 \to f^*\Omega^1_Y \otimes \omega_{X/Y}^* \to \Omega^1_X \otimes
  \omega_{X/Y}^* \to \underbrace{\Omega^1_{X/Y} \otimes \omega_{X/Y}^*}_{\cong
    \mathcal O_X} \to 0.
  $$
  Using that $f_* (\mathcal O_X) = \mathcal O_Y$, the first connecting
  morphism associated with this sequence then reads
  \begin{equation}\label{eq:X1}
    \mathcal O_Y \to \Omega^1_Y \otimes R^1f_*(\omega_{X/Y}^*) =: \mathcal F.    
  \end{equation}
  The sheaf $\mathcal F$ is naturally isomorphic to the sheaf Hom$\bigl( T_Y,
  R^1f_*(T_{X/Y}) \bigr)$. To give a morphism $\mathcal O_Y \to \mathcal F$ is
  thus the same as to give a map $T_Y \to R^1f_*(T_{X/Y})$, and the morphism
  obtained in~\eqref{eq:X1} is the same as the Kodaira-Spencer map discussed
  in Theorem~\ref{thm:Kodaira-Spencer}.

  Observe also that Serre duality yields a natural identification of $\mathcal
  F$ with the sheaf Hom$\bigl(f_*(\omega_{X/Y}^{\otimes 2}), \Omega^1_Y
  \bigr)$. To give a morphism $\mathcal O_Y \to \mathcal F$ it is thus the
  same as to give a map $f_*(\omega_{X/Y}^{\otimes 2}) \to \Omega^1_Y$. The
  morphism obtained in this way from~\eqref{eq:X1} is of course the morphism
  $\kappa^*$ of Equation~\eqref{eq:X2}.
\end{construction}

\subsubsection{Proof of Theorem~\ref*{thm:VZ}, construction of the $\boldsymbol{\tau^0_{p,q}}$}
\label{sssec:hKS1}

In the general setting of Theorem~\ref{thm:VZ} where the simplifying
Assumptions~\ref{sass:famcurves} do not generally hold, the starting point of
the Viehweg-Zuo construction is the standard sequence of relative logarithmic
differentials associated to the flat morphism $f'$ which generalises
Sequence~\eqref{eq:reldiff} from above,
\begin{equation}\label{eq:lds}
  0 \to (f')^*\Omega^1_{Y'}(\log D') \to \Omega^1_{X'}(\log \Delta') \to
  \Omega^1_{X'/Y'}(\log \Delta') \to 0.
\end{equation}
We refer to \cite[Sect.~4]{EV90} for a discussion of Sequence~\eqref{eq:lds},
and for a proof of the fact that the cokernel $\Omega^1_{X'/Y'}(\log \Delta')$
is locally free. By \cite[II, Ex.~5.16]{Ha77}, Sequence~\eqref{eq:lds} defines
a filtration of the $p^{\rm th}$ exterior power,
$$
\Omega^p_{X'}(\log \Delta') = F^0 \supseteq F^1 \supseteq \cdots \supseteq F^p
\supseteq F^{p+1} = \{0\},
$$
with $F^r/F^{r+1} \cong (f')^*\bigl( \Omega^r_{Y'}(\log D') \bigr) \otimes
\Omega^{p-r}_{X'/Y'}(\log \Delta')$. Take the first sequence induced by the
filtration,
$$
0 \longrightarrow F^1 \longrightarrow F^0 \longrightarrow F^0/F^1
\longrightarrow 0,
$$
modulo $F^2$, and obtain
\begin{multline}\label{eq:amt}
  0 \longrightarrow (f')^*\bigl( \Omega^1_{Y'}(\log D')\bigr) \otimes
  \Omega^{p-1}_{X'/Y'}(\log \Delta')
  \longrightarrow F^0/F^2 \longrightarrow \\
  \longrightarrow \Omega^p_{X'/Y'}(\log \Delta') \longrightarrow 0.
\end{multline}
Setting $\mathcal L := \Omega^n_{X'/Y'} (\log \Delta')$, twisting
Sequence~\eqref{eq:amt} with $\mathcal L^{-1}$ and pushing down, the connecting
morphisms of the associated long exact sequence give maps
$$
\tau^0_{p,q}: F^{p,q} \longrightarrow F^{p-1,q+1} \otimes \Omega_{Y'}^1(\log
D'),
$$
where $F^{p,q}:= R^q f'_*( \Omega_{X'/Y'}^p (\log \Delta') \otimes \mathcal
L^{-1})/\text{torsion}$. Set $\mathcal N_0^{p,q} := \ker (\tau^0_{p,q})$.

\subsubsection{Alignment of the $\boldsymbol{\tau^0_{p,q}}$}
\label{sssec:hKS2}

The morphisms $\tau^0_{p,q}$ and $\tau^0_{p-1,q+1}$ can be composed if we
tensor the latter with the identity morphism on $\Omega_{Y'}^1(\log D')$. More
specifically, we consider the following morphisms,
$$
\underbrace{\tau^0_{p,q} \otimes \Id_{\Omega^1_{Y'}(\log D')^{\otimes q}}}_{=:
  \tau_{p,q}} : F^{p,q} \otimes \bigl(\Omega^1_{Y'}(\log D')\bigr)^{\otimes q}
\to F^{p-1,q+1} \otimes \bigl( \Omega_{Y'}^1(\log D') \bigr)^{\otimes q+1},
$$
and their compositions
\begin{equation}\label{eq:tk}
  \underbrace{\tau_{n-k+1,k-1} \circ \cdots \circ \tau_{n-1,1} \circ \tau_{n,0}}_{=:
    \tau^k} : F^{n,0} \to F^{n-k, k} \otimes \bigl(\Omega_{Y'}^1(\log D')
  \bigr)^{\otimes k}.
\end{equation}

\subsubsection{Fundamental facts about $\boldsymbol{\tau^k}$ and $\boldsymbol{\mathcal N_0^{p,q}}$}
\label{ssec:facts}

Theorem~\ref{thm:VZ} is shown by relating the morphism $\tau^0_{p,q}$ with the
structure morphism of a Higgs-bundle coming from the variation of Hodge
structures associated with the family $f^\circ$. Viehweg's positivity results
of push-forward sheaves of relative differentials, as well as Zuo's results on
the curvature of kernels of generalised Kodaira-Spencer maps are the main
input here. Rather than recalling the complicated line of argumentation, we
simply state two central results from the argumentation of \cite{VZ02}.

\begin{fact}[Factorization via symmetric differentials, \protect{\cite[Lem.~4.6]{VZ02}}]\label{fact:VZ46}
  For any $k$, the morphism $\tau^k$ factors via the symmetric differentials
  $\Sym^k \Omega_{Y'}^1(\log D') \subseteq \bigl(\Omega_{Y'}^1(\log D')
  \bigr)^{\otimes k}$. More precisely, the morphism $\tau^k$ takes its image
  in $F^{n-k, k} \otimes \Sym^k \Omega_{Y'}^1(\log D')$.  \qed
\end{fact}

\begin{cons}
  Using Fact~\ref{fact:VZ46} and the observation that $F^{n,0} \cong
  \sO_{Y'}$, we can therefore view $\tau^k$ as a morphism
  $$
  \tau^k : \sO_{Y'} \longrightarrow F^{n-k, k} \otimes \Sym^k
  \Omega_{Y'}^1(\log D').
  $$
\end{cons}

While the proof of Fact~\ref{fact:VZ46} is rather elementary, the following
deep result is at the core of Viehweg-Zuo's argument. Its role in the proof of
Theorem~\ref{thm:VZ} is comparable to that of the Positivity
Theorem~\ref{thm:Vpos} discussed in Section~\ref{ssec:8B}.

\begin{fact}[Negativity of $\mathcal N_0^{p,q}$, \protect{\cite[Claim~4.8]{VZ02}}]\label{fact:VZ48}
  Given any numbers $p$ and $q$, there exists a number $k$ and an invertible
  sheaf $\mathcal A' \in \Pic(Y')$ of Kodaira-Iitaka dimension $\kappa(\mathcal A') \geq
  \Var(f^0)$ such that $(\mathcal A')^* \otimes \Sym^k \bigl( (\mathcal N_0^{p,q})^*
  \bigr)$ is generically generated. \qed
\end{fact}

\subsubsection{End of proof}
\label{ssec:eopVZ}

To end the sketch of proof, we follow \cite[p.~315]{VZ02} almost verbatim. By
Fact~\ref{fact:VZ48}, the trivial sheaf $F^{n,0} \cong \sO_{Y'}$ cannot lie in
the kernel $\mathcal N_0^{n,0}$ of $\tau^1 = \tau^0_{n,0}$. We can therefore
set $1 \leq m$ to be the largest number with $\tau^{m}(F^{n,0}) \not = \{0\}$.
Since $m$ is maximal, $\tau^{m+1} = \tau_{n-m,m} \circ \tau^{m} \equiv 0$ and
$$
\Image(\tau^{m}) \subseteq \ker(\tau_{n-m,m}) = \mathcal N_0^{n-m,m}\otimes \Sym^m
\Omega^1_{Y'}(\log D').
$$
In other words, $\tau^{m}$ gives a non-trivial map
$$
\tau^{m}: \sO_{Y'} \cong F^{n,0} \longrightarrow \mathcal N_0^{n-m,m}\otimes \Sym^m
\Omega^1_{Y'}(\log D').
$$
Equivalently, we can view $\tau^{m}$ as a non-trivial map
\begin{equation}\label{eq:end}
  \tau^{m}: (\mathcal N_0^{n-m,m})^*  \longrightarrow \Sym^m \Omega^1_{Y'}(\log D').
\end{equation}
By Fact~\ref{fact:VZ48}, there are many morphisms $\mathcal A' \to \Sym^k \bigl(
(\mathcal N_0^{n-m, m})^* \bigr)$, for $k$ large enough. Together
with~\eqref{eq:end}, this gives a non-zero morphism $\mathcal A' \to \Sym^{k\cdot m}
\Omega^1_{Y'}(\log D')$.

We have seen in Remark~\ref{rem:extoS} that the sheaf $\mathcal A' \subseteq
\Sym^{k\cdot m} \Omega^1_{Y'}(\log D')$ extends to a sheaf $\mathcal A \subseteq
\Sym^{k\cdot m} \Omega^1_Y(\log D)$ with $\kappa(\mathcal A) = \kappa(\mathcal A') \geq
\Var(f^\circ)$. This ends the proof of Theorem~\ref{thm:VZ}. \qed

\subsection{Open problems}
 	
In spite of its importance, little is known about further properties that the
Viehweg-Zuo sheaves $\mathcal A$ might have.

\begin{question}
  For families of higher-dimensional manifolds, how does the Viehweg-Zuo
  construction behave under base change? Does it satisfy any universal
  properties at all? If not, is there a ``natural'' positivity result for base
  spaces of families that does satisfy good functorial properties?
\end{question}

In the setup of Theorem~\ref{thm:VZ}, if $Z^\circ \subset Y^\circ$ is any
closed submanifold, then the associated Viehweg-Zuo sheaves $\mathcal A$,
constructed for the family $f^\circ : X^\circ \to Y^\circ$, and $\mathcal
A_Z$, constructed for the restricted family $f^\circ_Z : X^\circ
\times_{Y^\circ} Z^\circ \to Z^\circ$, may differ. In particular, it is not
clear that $\mathcal A_Z$ is the restriction of $\mathcal A$, and the sheaves
$\mathcal A$ and $\mathcal A_Z$ may live in different symmetric products of
their respective $\Omega^1$'s.

One likely source of non-compatibility with base change is the choice of the
number $m$ in Section~\ref{ssec:eopVZ} (``largest number with
$\tau^{m}(F^{n,0}) \not = \{0\}$''). It seems unlikely that this definition
behaves well under base change.

\begin{question}
  For families of higher-dimensional manifolds, are there distinguished
  subvarieties in moduli space that have special Viehweg-Zuo sheaves, perhaps
  contained in particularly high/low symmetric powers of $\Omega^1$? Does the
  lack of a restriction morphism induce a geometric structure on the moduli
  space?
\end{question}

The refinement of the Viehweg-Zuo Theorem, presented in
Theorem~\ref{thm:VZimproved} above, turns out to be important for the
applications that we have in mind. It is, however, not clear to us if the
sheaf $\mathcal B$ which appears in Theorem~\ref{thm:VZimproved} is really
optimal.

\begin{question}
  Prove that the sheaf $\mathcal B \subseteq \Omega^1_Y(\log D)$ is the
  smallest sheaf possible for which Theorem~\ref{thm:VZimproved} holds, or
  else find the smallest possible sheaf. For instance, does
  Theorem~\ref{thm:VZimproved} admit a natural improvement if we replace
  $\Omega^1_Y(\log D)$ by a suitable sheaf of orbifold differentials, using
  Campana's language of geometric orbifolds?
\end{question}

\section{Techniques II: Reflexive differentials on singular spaces}
\label{sec:ext}

\subsection{Motivation}
\label{ssec:motiv}

\subsubsection{A special case of the Viehweg conjecture}

To motivate the results presented in this section, let $f^\circ : X^\circ \to
Y^\circ$ be a smooth, projective family of canonically polarised varieties
over a smooth, quasi-projective base manifold, and assume that the family
$f^\circ$ is of maximal variation, i.e., that $\Var(f^\circ) = \dim
Y^\circ$. As before, choose a smooth compactification $Y \supseteq Y^\circ$
such that $D := Y \setminus Y^\circ$ is a divisor with only simple normal
crossings.

To prove Viehweg's conjecture, we need to show that the logarithmic Kodaira
dimension of $Y^\circ$ is maximal, i.e., that $\kappa(Y^\circ) = \dim
Y^\circ$. In particular, we need to rule out the possibility that
$\kappa(Y^\circ)=0$. As we will see in the proof of
Proposition~\ref{prop:VZec1} below, a relatively elementary argument exists in
cases where the Picard number of $Y$ is one, $\rho(Y)=1$. We refer the reader
to \cite[Sect.~I.1]{HL97} for the notion of semistability and for a discussion
of the Harder-Narasimhan filtration used in the proof.

\begin{myprop}[Partial answer to Viehweg's conjecture in case $\rho(Y)=1$]\label{prop:VZec1}
  In the setup described above, if we additionally assume that $\rho(Y)=1$,
  then $\kappa(Y^\circ) \not = 0$.
\end{myprop}
\begin{proof}
  We argue by contradiction and assume that both $\kappa(Y^\circ) = 0$ and
  that $\rho(Y)=1$.  Let $\mathcal A \subseteq \Sym^m \Omega^1_Y(\log D)$ be
  the big invertible sheaf whose existence is guaranteed by the Viehweg-Zuo
  construction, Theorem~\ref{thm:VZ}. Since $\rho(Y)=1$, the sheaf $\sA$ is
  actually ample.

  As a first step, observe that the log canonical bundle $K_Y+D$ must be
  torsion, i.e., that there exists a number $m' \in \mathbb N^+$ such that
  $\mathcal O_Y\bigl(m' \cdot (K_Y+D) \bigr) \cong \mathcal O_Y$. This follows
  from the assumption that $\kappa(K_Y+D) = 0$ and from the observation that
  on a projective manifold with $\rho=1$, any invertible sheaf which admits a
  non-zero section is either trivial or ample. In particular, we obtain that
  the divisor $K_Y+D$ is numerically trivial.

  Next, let $C \subseteq Y$ be a general complete intersection curve in the
  sense of Mehta-Ramanathan, cf.~\cite[Sect.~II.7]{HL97}. The numerical
  triviality of $K_Y+D$ will then imply that
  $$
  (K_Y+D).C = c_1 \bigl( \Omega^1_Y(\log D) \bigr).C = c_1\bigl(\Sym^m
  \Omega^1_Y(\log D) \bigr).C = 0.
  $$
  On the other hand, since $\sA$ is ample, we have that $c_1(\mathcal A).C >
  0$. In summary, we obtain that the symmetric product sheaf $\Sym^m
  \Omega^1_Y(\log D)$ is not semistable. Since we are working in
  characteristic zero, this implies that the sheaf of Kähler differentials
  $\Omega^1_Y(\log D)$ will likewise not be semistable, and contains a
  destabilising subsheaf $\mathcal B \subseteq \Omega^1_Y(\log D)$ with
  $c_1(\mathcal B).C > 0$, cf.~\cite[Cor.~3.2.10]{HL97}. Since the
  intersection number $c_1(\mathcal B).C$ is positive, the rank $r$ of the
  sheaf $\mathcal B$ must be strictly less than $\dim Y$, and its determinant
  is an ample invertible subsheaf of the sheaf of logarithmic $r$-forms,
  $$
  \det \mathcal B \subseteq \Omega^r_Y(\log D).
  $$
  This, however, contradicts the Bogomolov-Sommese Vanishing
  Theorem~\ref{thm:classBSv} and therefore ends the proof.
\end{proof}

\subsubsection{Application of minimal model theory}

The assumption that $\rho(Y)=1$ is not realistic. In the general situation,
where $\rho(Y)$ can be arbitrarily large, we will apply the minimal model
program to the pair $(Y,D)$. As we will sketch in
Section~\ref{sec:ideaOfProof}, assuming that the standard conjectures of
minimal model theory hold true, a run of the minimal model program for a
suitable choice of a boundary divisor will yield a diagram,
$$
\xymatrix{ %
  Y \ar@{-->}[rrr]^{\lambda}_{\text{minimal model program}} &&& Y_\lambda
  \ar[d]^{\text{fibre space}}_{\pi} \\
  &&& Z_\lambda,
  }
$$
where $\lambda : Y \dasharrow Y_\lambda$ is a birational map whose inverse
does not contract any divisors, and where either $\rho(Y_\lambda) =1$ and
$Z_\lambda$ is a point, or where $Y_\lambda$ has the structure of a proper
Mori-- or Kodaira fibre space. In the first case, we can try to copy the proof
of Proposition~\ref{prop:VZec1} above. In the second case, we can use the
fibre structure and try to argue inductively.

The main problem that arises when adopting the proof of
Proposition~\ref{prop:VZec1} is the presence of singularities. Both the space
$Y_\lambda$ and the cycle-theoretic image divisor $D_\lambda \subset Y_\lambda$
will generally be singular, and the pair $(Y_\lambda, D_\lambda)$ will generally
be dlt. This leads to two difficulties.

\begin{enumerate}
\item The sheaf $\Omega^1_{Y_\lambda}(\log D_\lambda)$ of logarithmic Kähler
  differentials is generally not pure in the sense of
  \cite[Sect.~1.1]{HL97}. Accordingly, there is no good notion of stability
  that would be suitable to construct a Harder-Narasimhan filtration.

\item\label{X} The Viehweg-Zuo construction does not work for singular
  varieties. The author is not aware of any method suitable to prove
  positivity results for Kähler differentials, or prove the existence of
  sections in any symmetric product of $\Omega^1_{Y_\lambda}(\log D_\lambda)$.
\end{enumerate}

The aim of the present Section~\ref{sec:ext} is to show that that both
problems can be overcome if we replace the sheaf $\Omega^1_{Y_\lambda}(\log
D_\lambda)$ of Kähler differentials by its double dual
$\Omega^{[1]}_{Y_\lambda}(\log D_\lambda) := \bigl( \Omega^1_{Y_\lambda}(\log
D_\lambda) \bigr)^{**}$. We refer to \cite[Sect.~1.6]{Reid87} for a discussion
of the double dual in this context, and to \cite[II.~Sect.~1.1]{OSS} for a
thorough discussion of reflexive sheaves. The following notation will be
useful in the discussion.

\begin{notation}[Reflexive tensor operations]\label{not:relfxive}
  Let $X$ be a normal variety and $\mathcal A$ a coherent sheaf of $\mathcal
  O_X$-modules. Given any number $n\in \mathbb N$, set $\mathcal A^{[n]} :=
  (\mathcal A^{\otimes n})^{**}$, $\Sym^{[n]} \mathcal A := (\Sym^n \mathcal
  A)^{**}$.  If $\pi: X' \to X$ is a morphism of normal varieties, set
  $\pi^{[*]}(\mathcal A) := \bigl( \pi^*\mathcal A \bigr)^{**}$. In a similar
  vein, set $\Omega^{[p]}_X := \bigl( \Omega^p_X \bigr)^{**}$ and
  $\Omega^{[p]}_X(\log D) := \bigl( \Omega^p_X (\log D) \bigr)^{**}$.
\end{notation}

\begin{notation}[Reflexive differential forms]
  A section in $\Omega^{[p]}_X$ or $\Omega^{[p]}_X(\log D)$ will be called a
  \emph{reflexive form} or a \emph{reflexive logarithmic form}, respectively.
\end{notation}

\begin{fact}[Torsion freeness and Harder-Narasimhan filtration]
  Reflexive sheaves are torsion free and therefore pure. In particular, a
  Harder-Narasimhan filtration exists for $\Omega^{[p]}_X(\log D)$ and for the
  symmetric products $\Sym^{[n]} \Omega^1_X(\log D)$.
\end{fact}

\begin{fact}[Extension over small sets]\label{fact:extension1}
  If $X$ is a normal space, if $\mathcal A$ is any reflexive sheaf on $X$ and
  if $Z \subset X$ any set of $\codim_X Z \geq 2$, then the restriction map
  $$
  H^0 \bigl( X,\, \mathcal A \bigr) \to H^0 \bigl( X\setminus Z,\, \mathcal A
  \bigr)
  $$
  is in fact isomorphic. We say that ``sections in $\mathcal A$ extend over the
  small set $Z$''.

  If $U := X \setminus Z$ is the complement of $Z$, with inclusion map $\iota
  : U \to X$, it follows immediately that $\mathcal A = \iota_*(\mathcal
  A|_U)$. In a similar vein, if $\mathcal B_U$ is any locally free sheaf on
  $U$, its push-forward sheaf $\iota_*(\mathcal B_U)$ will always be
  reflexive.
\end{fact}

\subsubsection{Outline of this section}

It follows almost by definition that sheaves of reflexive differentials have
very good push-forward properties. In Section~\ref{ssec:pfwd} we will use
these properties to overcome one of the problems mentioned above and to
produce Viehweg-Zuo sheaves of reflexive differentials on singular
spaces. Perhaps more importantly, we will in Section~\ref{ssec:pb} recall
extension results for log canonical varieties. These results show that
reflexive differentials often admit a pull-back map, similar to the standard
pull-back of Kähler differentials. A generalisation of the Bogomolov-Sommese
vanishing theorem to log canonical varieties follows as a corollary. 

In Section~\ref{ssec:res}, we recall that some of the most important
constructions known for logarithmic differentials on snc pairs also work for
reflexive differentials on dlt pairs. This includes the existence of a residue
sequence. For our purposes, this makes reflexive differentials almost as
useful as regular differentials in the theory of smooth spaces. As we will
roughly sketch in Section~\ref{sec:ideaOfProof}, these results will allow to
adapt the proof of Proposition~\ref{prop:VZec1} to the singular setup, and
will give a proof of Viehweg's Conjecture~\ref{conj:Viehweg}, at least for
families over base manifolds of dimension $\leq 3$. Section~\ref{ssec:IOPpb}
gives a brief sketch of the proof of the pull-back result of
Section~\ref{ssec:pb}. We end by mentioning a few open problems and
conjectures.

Some of the material presented in the current Section~\ref{sec:ext}, including
Section~\ref{ssec:res} and all the illustrations, is taken without much
modification from the paper \cite{GKKP10}. Section~\ref{ssec:pfwd} follows the
paper \cite{KK08c}.

\subsection{Existence of a push-forward map}
\label{ssec:pfwd}

Fact~\ref{fact:extension1} implies that any Viehweg-Zuo sheaf which exists on
a pair $(Z, \Delta)$ of a smooth variety and a reduced divisor with simple
normal crossing support immediately implies the existence of a Viehweg-Zuo
sheaf of reflexive differentials on any minimal model of $(Z, \Delta)$, and
that the Kodaira-Iitaka dimension only increases in the process.  To formulate
the result precisely, we briefly recall the definition of the Kodaira-Iitaka
dimension for reflexive sheaves.

\begin{defn}[\protect{Kodaira-Iitaka dimension of a sheaf, \cite[Not.~2.2]{KK08c}}]\label{def:KIdim}
  Let $Z$ be a normal projective variety and $\sA$ a reflexive sheaf of rank
  one on $Z$.  If $h^0\bigl(Z,\, \sA^{[n]}\bigr) = 0$ for all $n \in \mathbb
  N$, then we say that $\sA$ has Kodaira-Iitaka dimension $\kappa(\sA) :=
  -\infty$.  Otherwise, set
  $$
  M := \bigl\{ n\in \mathbb N \,|\, h^0\bigl(Z,\, \sA^{[n]}\bigr)>0\bigr\},
  $$
  recall that the restriction of $\sA$ to the smooth locus of $Z$ is locally
  free and consider the natural rational mapping
  $$
  \phi_n : Z \dasharrow \mathbb P\bigl(H^0\bigl(Z,\, \sA^{[n]}\bigr)^*\bigr)
  \text{ for each } n \in M.
  $$
  The Kodaira-Iitaka dimension of $\sA$ is then defined as
  $$
  \kappa(\sA) := \max_{n \in M} \bigl(\dim \overline{\phi_n(Z)}\bigr).
  $$
\end{defn}

With this notation, the main result concerning the push-forward is then
formulated as follows.

\begin{myprop}[\protect{Push forward of Viehweg-Zuo sheaves, \cite[Lem.~5.2]{KK08c}}]\label{lem:pushdownA}
  Let $(Z, \Delta)$ be a pair of a smooth variety and a reduced divisor with
  simple normal crossing support. Assume that there exists a reflexive sheaf
  $\sA \subseteq \Sym^{[n]} \Omega^1_Z(\log \Delta)$ of rank one. If $\lambda
  : Z \dasharrow Z'$ is a birational map whose inverse does not contract any
  divisor, if $Z'$ is normal and $\Delta'$ is the (necessarily reduced)
  cycle-theoretic image of $\Delta$, then there exists a reflexive sheaf $\sA'
  \subseteq \Sym^{[n]} \Omega^1_{Z'}(\log \Delta')$ of rank one, and of
  Kodaira-Iitaka dimension $\kappa(\sA')\geq \kappa(\sA)$.
\end{myprop}
\begin{proof}
  The assumption that $\lambda^{-1}$ does not contract any divisors and the
  normality of $Z'$ guarantee that $\lambda^{-1}: Z' \dasharrow Z$ is a
  well-defined embedding over an open subset $U \subseteq Z'$ whose complement
  has codimension $\codim_{Z'} (Z' \setminus U) \geq 2$, cf.~Zariski's main
  theorem \cite[V~5.2]{Ha77}.  In particular, $\Delta'|_U = \bigl(
  \lambda^{-1}|_U \bigr)^{-1}(\Delta)$. Let $\iota: U \to Z'$ denote the
  inclusion and set $\mathcal A' := \iota_* \bigl(
  (\lambda^{-1}|_U)^{*}\mathcal A \bigr)$ ---this sheaf is reflexive by
  Fact~\ref{fact:extension1}. We obtain an inclusion of reflexive sheaves,
  $\sA' \subseteq \Sym^{[n]} \Omega^1_{Z'}(\log \Delta')$. By construction, we
  have that $h^0\bigl(Z',\, \sA'^{[m]}\bigr) \geq h^0(Z,\, \sA^{[m]})$ for all
  $m>0$, hence $\kappa(\sA') \geq \kappa(\sA)$.
\end{proof}

Given the importance of the Viehweg-Zuo construction, Theorem~\ref{thm:VZ}, we
will call the sheaves $\mathcal A$ which appear in
Proposition~\ref{lem:pushdownA} ``Viehweg-Zuo sheaves''.

\begin{notation}[Viehweg-Zuo sheaves]
  Let $(Z, \Delta)$ be a pair of a smooth variety and a reduced divisor with
  simple normal crossing support, and let $n \in \mathbb N$ be any number. A
  reflexive sheaf $\sA \subseteq \Sym^{[n]} \Omega^1_Z(\log \Delta)$ of rank
  one will be called a ``Viehweg-Zuo sheaf''.
\end{notation}

\subsection{Existence of a pull-back morphism, statement and applications}
\label{ssec:pb}

Kähler differentials are characterised by a number of universal properties,
one of the most important being the existence of a pull-back map: if $\gamma :
Z \to X$ is any morphism of algebraic varieties and if $p \in \mathbb N$, then
there exists a canonically defined sheaf morphism
\begin{equation}\label{eq:pbKd}
  d\gamma : \gamma^* \Omega^p_X  \to \Omega^p_Z.
\end{equation}

The following example illustrates that for sheaves of reflexive differentials on
normal spaces, a pull-back map does not exist in general.

\begin{example}[\protect{Pull-back morphism for dualizing sheaves, cf.~\cite[Ex.~4.2]{GKKP10}}]
  Let $X$ be a normal Gorenstein variety of dimension $n$, and let $\gamma : Z
  \to X$ be any resolution of singularities.  Observing that the sheaf of
  reflexive $n$-forms is precisely the dualizing sheaf, $\Omega^{[n]}_X \simeq
  \omega_X$, it follows directly from the definition of canonical
  singularities that $X$ has canonical singularities if and only if a
  pull-back morphism $d\gamma : \gamma^* \Omega^{[n]}_X \to \Omega^{n}_Z$
  exists.
\end{example}

Together with Daniel Greb, Sándor Kovács and Thomas Peternell, the author has
shown that a pull-back map for reflexive differentials always exists if the
target is log canonical.

\begin{mythm}[\protect{Pull-back map for differentials on lc pairs, \cite[Thm.~4.3]{GKKP10}}]\label{thm:generalpullback}
  Let $(X, D)$ be an log canonical pair, and let $\gamma : Z \to X$ be a
  morphism from a normal variety $Z$ such that the image of $Z$ is not
  contained in the reduced boundary or in the singular locus, i.e.,
  $$
  \gamma(Z) \not\subseteq (X,D)_{\rm sing} \cup \supp \lfloor D \rfloor.
  $$
  If $1 \leq p \leq \dim X$ is any index and
  $$
  \Delta := \text{largest reduced Weil divisor contained in }
  \gamma^{-1}\bigl(\text{non-klt locus}\bigr),
  $$
  then there exists a sheaf morphism,
  $$
  d\gamma : \gamma^* \Omega^{[p]}_X(\log \lfloor D \rfloor) \to
  \Omega^{[p]}_Z(\log \Delta),
  $$
  that agrees with the usual pull-back morphism \eqref{eq:pbKd} of Kähler
  differentials at all points $p \in Z$ where $\gamma(p) \not \in (X,D)_{\rm
    sing} \cup \supp \lfloor D \rfloor$.
\end{mythm}

\begin{rem}
  If follows from the definition of klt, \cite[Def.~2.34]{KM98}, that the
  components of $D$ which appear with coefficient one are always contained in
  the non-klt locus of $(X,D)$. In particular, the divisor $\Delta$ defined in
  Theorem~\ref{thm:generalpullback} always contains the codimension-one part
  of $\gamma^{-1}\bigl( \supp \lfloor D \rfloor \bigr)$.

\end{rem}

The assertion of Theorem~\ref{thm:generalpullback} is rather general and
perhaps a bit involved. For klt spaces, the statement reduces to the
following simpler result.

\begin{mythm}[\protect{Pull-back map for differentials on klt spaces}]
  Let $X$ be a normal klt variety\footnote{More precisely, we should say ``Let
    $X$ be a normal variety such that the pair $(X, \emptyset)$ is
    klt\ldots''}, and let $\gamma : Z \to X$ be a morphism from a normal
  variety $Z$ such that the image $\gamma(Z)$ is not entirely contained in the
  singular locus of $X$.  If $1 \leq p \leq \dim X$ is any index then there
  exists a sheaf morphism,
  $$
  d\gamma : \gamma^* \Omega^{[p]}_X \to \Omega^{[p]}_Z,
  $$
  that agrees on an open set with the usual pull-back morphism of Kähler
  differentials. \qed
\end{mythm}

Extension properties of differential forms that are closely related to the
existence of pull-back maps have been studied in the literature, mostly
considering only special values of $p$. Using Steenbrink's generalization of
the Grauert-Riemenschneider vanishing theorem as their main input, similar
results were shown by Steenbrink and van Straten for varieties $X$ with only
isolated singularities and for $p \leq \dim X-2$, without any further
assumption on the nature of the singularities, \cite[Thm.~1.3]{SS85}. Flenner
extended these results to normal varieties, subject to the condition that $p
\leq \codim X_{\rm sing} - 2$, \cite{Flenner88}. Namikawa proved the extension
properties for $p \in \{1, 2\}$, in case $X$ has canonical Gorenstein
singularities, \cite[Thm.~4]{Namikawa01}. In the case of finite quotient
singularities similar results were obtained in \cite{deJongStarr}. For a log
canonical pair with reduced boundary divisor, the cases $p \in \{1, \dim X-1,
\dim X\}$ were settled in \cite[Thm.~1.1]{GKK08}.

A related setup where the pair $(X,D)$ is snc, and where $\pi : \widetilde{X}
\to X$ is the composition of a finite Galois covering and a subsequent
resolution of singularities has been studied by Esnault and Viehweg. In
\cite{RevI} they obtain in their special setting similar results and
additionally prove vanishing of higher direct image sheaves.

A brief sketch of the proof of Theorem~\ref{thm:generalpullback} is given in
Section~\ref{ssec:IOPpb} below. The proof uses a strengthening of the
Steenbrink vanishing theorem, which follows from local Hodge-theoretic
properties of log canonical singularities, in particular from the fact that
log canonical spaces are Du~Bois. These methods are combined with results
available only for special classes of singularities, such as the recent
progress in minimal model theory and partial classification of singularities
that appear in minimal models.

\subsubsection{Applications}

Theorem~\ref{thm:generalpullback} has many applications useful for moduli
theory. We mention two applications which will be important in our
context. The first application generalises the Bogomolov-Sommese vanishing
theorem to singular spaces.

\begin{cor}[\protect{Bogomolov-Sommese vanishing for lc pairs, \cite[Thm.~7.2]{GKKP10}}]\label{cor:BS}
  Let $(X, D)$ be a log canonical pair, where $X$ is projective. If $\sA
  \subseteq \Omega^{[p]}_X(\log \lfloor D \rfloor)$ is a $\mathbb Q$-Cartier
  reflexive subsheaf of rank one, then $\kappa(\sA) \leq p$.
\end{cor}

\begin{rem}[Notation used in Corollary~\ref{cor:BS}]
  The number $\kappa(\sA)$ appearing in the statement of
  Corollary~\ref{cor:BS} is the generalised Kodaira-Iitaka dimension
  introduced in Definition~\ref{def:KIdim}. A reflexive sheaf $\sA$ is rank
  one is called $\mathbb Q$-Cartier if there exists a number $n \in \mathbb
  N^+$ such that the $n$th reflexive tensor product $\sA^{[n]}$ is invertible.
\end{rem}

\begin{proof}[Proof of Corollary~\ref{cor:BS} in a special case]
  We prove Corollary~\ref{cor:BS} only in the special case where the sheaf
  $\sA \subseteq \Omega^{[p]}_X(\log \lfloor D \rfloor)$ is invertible. The
  reader interested in a full proof is referred to the original reference
  \cite{GKKP10}.

  Let $\gamma: Z \to X$ be any resolution of singularities, and let $\Delta
  \subset Z$ be the reduced divisor defined in
  Theorem~\ref{thm:generalpullback} above.  Theorem~\ref{thm:generalpullback}
  will then assert the existence of an inclusion 
  $$
  \gamma^*(\sA) \to \Omega^p_Z(\log \Delta),
  $$
  and the standard Bogomolov-Sommese vanishing result,
  Theorem~\ref{thm:classBSv}, applies to give that $\kappa
  \bigl(\gamma^*(\sA)\bigr) \leq p$. Since $\sA$ is invertible, and since
  $\gamma$ is birational, it is clear that $\kappa \bigl(\gamma^*(\sA)\bigr) =
  \kappa (\sA)$, finishing the proof.
\end{proof}

\begin{warning}
  Taking the double dual of a sheaf does generally \emph{not} commute with
  pulling back. Since reflexive tensor products were used in
  Definition~\ref{def:KIdim} to define the Kodaira-Iitaka dimension of a
  sheaf, it is generally false that the Kodaira-Iitaka dimension stays
  invariant when pulling a sheaf $\sA$ back to a resolution of
  singularities. The proof of Corollary~\ref{cor:BS} which is given in the
  simple case where $\sA$ is invertible does therefore not work without
  substantial modification in the general setup where $\sA$ is only $\mathbb
  Q$-Cartier.
\end{warning}

The second application of Theorem~\ref{thm:generalpullback} concerns
rationally chain connected singular spaces. Rationally chain connected
manifolds are rationally connected, and do not carry differential
forms. Building on work of Hacon and McKernan, \cite{HMcK07}, we show that the
same holds for reflexive forms on klt pairs.

\begin{cor}[\protect{Reflexive differentials on rationally chain connected spaces, \cite[Thm.~5.1]{GKKP10}}]\label{cor:kltRCC}
  Let $(X,D)$ be a klt pair. If $X$ is rationally chain connected, then $X$ is
  rationally connected, and $H^0 \bigl( X,\, \Omega^{[p]}_X \bigr) = 0 $ for
  all $p\in\mathbb N, 1 \leq p \leq \dim X$.
\end{cor}
\begin{proof}
  Choose a log resolution of singularities, $\pi: \widetilde{X} \to X$ of the
  pair $(X,D)$. Since klt pairs are also dlt, a theorem of Hacon-McKernan,
  \cite[Cor.~1.5(2)]{HMcK07}, applies to show that $X$ and $\widetilde{X}$ are
  both rationally connected. In particular, it follows that $H^0 \bigl(
  \widetilde{X},\, \Omega^p_{\widetilde{X}}) = 0$ for all $p > 0$ by
  \cite[IV.~Cor.~3.8]{K96}.

  Since $(X,D)$ is klt, Theorem~\ref{thm:generalpullback} asserts that there
  exists a pull-back morphism $d\pi : \pi^* \Omega^{[p]}_X \to
  \Omega^p_{\widetilde{X}}$.  As $\pi$ is birational, $d\pi$ is generically
  injective and since $\Omega^{[p]}_X$ is torsion-free, this means that the
  induced morphism on the level of sections is injective:
  $$
  \pi^*: H^0 \bigl( X,\, \Omega^{[p]}_X \bigr) \to H^0 \bigl( \widetilde{X},\,
  \Omega^p_{\widetilde{X}} \bigr) = 0.
  $$
  The claim then follows.
\end{proof}

\subsection{Residue theory and restrictions for differentials on dlt pairs}
\label{ssec:res}

Logarithmic Kähler differentials on snc pairs are canonically defined. They
are characterised by strong universal properties and appear accordingly in a
number of important sequences, filtered complexes and other
constructions. First examples include the following:%

\begin{enumerate}
\item the pull-back property of differentials under arbitrary morphisms,

\item relative differential sequences for smooth morphisms,

\item residue sequences associated with snc pairs, and

\item the description of Chern classes as the extension classes of the first
  residue sequence.
\end{enumerate}

Reflexive differentials do in general not enjoy the same universal properties
as Kähler differentials. However, we have seen in
Theorem~\ref{thm:generalpullback} that reflexive differentials do have good
pull-back properties if we are working with log canonical pairs. In the
present Section~\ref{ssec:res}, we would like to make the point that each of
the other properties listed above also has a very good analogue for reflexive
differentials, as long as we are working with dlt pairs. This makes reflexive
differential extremely useful in practise. In a sense, it seems fair to say
that ``reflexive differentials and dlt pairs are made for one another''.

\subsubsection{The relative differential sequence for snc pairs}
\label{sec:relReflSeqSNC}

Here we recall the generalisation of the standard sequence for relative
differentials, \cite[Prop.~II.8.11]{Ha77}, to the logarithmic setup. For this,
we introduce the notion of an \emph{snc morphism} as a logarithmic analogue of
a smooth morphism.  Although \emph{relatively snc divisors} have long been
used in the literature, cf.~\cite[Sect.~3]{Deligne70}, we are not aware of a
good reference that discusses them in detail. Recall that a pair $(X,D)$ is
called an ``snc pair'' if $X$ is smooth, and if the divisor $D$ is reduced and
has only simple normal crossing support.

\begin{notation}[Intersection of boundary components]\label{not:DI}
  Let $(X,D)$ be a pair of a normal space $X$ and a divisor $D$, where $D$ is
  written as a sum of its irreducible components $D = \alpha_1 D_1 + \ldots +
  \alpha_n D_n$.  If $I \subseteq \{1, \ldots, n\}$ is any non-empty subset,
  we consider the scheme-theoretic intersection $D_I := \cap_{i \in I}
  D_i$. If $I$ is empty, set $D_I := X$.
\end{notation}

\begin{rem}[Description of snc pairs]\label{rem:descrSNC}
  In the setup of Notation~\ref{not:DI}, it is clear that the pair $(X,D)$ is
  snc if and only if all $D_I$ are smooth and of codimension equal to the
  number of defining equations: $\codim_XD_I=|I|$ for all $I$ where $D_I \not
  = \emptyset$.
\end{rem}

\begin{defn}[\protect{Snc morphism, relatively snc divisor, \cite[Def.~2.1]{VZ02}}]\label{def:sncMorphism}
  If $(X,D)$ is an snc pair and $\phi: X \to T$ a surjective morphism to a
  smooth variety, we say that $D$ is \emph{relatively snc}, or that $\phi$ is
  \emph{an snc morphism of the pair} $(X,D)$ if for any set $I$ with $D_I \not
  = \emptyset$ all restricted morphisms $\phi|_{D_I} : D_I \to T$ are smooth
  of relative dimension $\dim X-\dim T -|I|$.
\end{defn}

\begin{rem}[Fibers of an snc morphisms]\label{rem:fiberSNC}
  If $(X,D)$ is an snc pair and $\phi: X \to T$ is any surjective snc morphism
  of $(X,D)$, it is clear from Remark~\ref{rem:descrSNC} that if $t \in T$ is
  any point, with preimages $X_t := \phi^{-1}(t)$ and $D_t := D \cap X_t$ then
  the pair $(X_t, D_t)$ is again snc.
\end{rem}

\begin{rem}[All morphisms are generically snc]\label{rem:genSNC}
  If $(X,D)$ is an snc pair and $\phi: X \to T$ is any surjective morphism, it
  is clear from generic smoothness that there exists a dense open set $T^\circ
  \subseteq T$, such that $D \cap \phi^{-1}(T^\circ)$ is relatively snc over
  $T^\circ$.
\end{rem}

Let $(X,D)$ be a reduced snc pair, and $\phi: X \to T$ an snc morphism of
$(X,D)$, as introduced in Definition~\ref{def:sncMorphism}. In this setting,
the standard pull-back morphism of 1-forms extends to yield the following
exact sequence of locally free sheaves on $X$,
\begin{equation}\label{eq:relDiff}
  0 \to \phi^*\Omega^1_T \to \Omega^1_X(\log D) \to \Omega^1_{X/T}(\log D) \to 0,
\end{equation}
called the ``relative differential sequence for logarithmic differentials''.
We refer to \cite[Sect.~4.1]{EV90} \cite[Sect.~3.3]{Deligne70} or
\cite[p.~137ff]{MR1924513} for a more detailed explanation. For forms of
higher degrees, the sequence \eqref{eq:relDiff} induces filtration
\begin{equation}\label{eq:relDiffFilt}
  \Omega^{p}_{X}(\log D) = \mathcal F^{0}(\log) \supseteq \mathcal F^1(\log) \supseteq \dots
  \supseteq \mathcal F^{p}(\log) \supseteq \{0\}
\end{equation}
with quotients
\begin{equation}\label{eq:relDiffQ}
  0 \to \mathcal F^{r+1}(\log) \to \mathcal F^r(\log) \to \phi^*\Omega_T^r\otimes
  \Omega_{X/T}^{p-r}(\log D) \to 0
\end{equation}
for all $r$. We refer to \cite[Ex.~II.5.16]{Ha77} for the construction
of~\eqref{eq:relDiffFilt}.

The main result of this section,
Theorem~\ref{thm:relativedifferentialfiltration}, gives analogues
of~\eqref{eq:relDiff}--\eqref{eq:relDiffQ} in case where $(X,D)$ is dlt. In
essence, Theorem~\ref{thm:relativedifferentialfiltration} says that all
properties of the relative differential sequence still hold on dlt pairs if
one removes from $X$ a set $Z$ of codimension $\codim_X Z \geq 3$.

\begin{mythm}[\protect{Relative differential sequence on dlt pairs, \cite[Thm.~10.6]{GKKP10}}]\label{thm:relativedifferentialfiltration}
  Let $(X, D)$ be a dlt pair with $X$ connected. Let $\phi: X \to T$ be a
  surjective morphism to a normal variety $T$.  Then, there exists a non-empty
  smooth open subset $T^\circ \subseteq T$ with preimages $X^\circ =
  \phi^{-1}(T^ \circ)$, $D^\circ = D \cap X^\circ $, and a filtration
  \begin{equation}\label{eq:relDiffFilt2}
    \Omega^{[p]}_{X^\circ}(\log \lfloor D^{\circ} \rfloor) = \mathcal F^{[0]}(\log)
    \supseteq \dots \supseteq \mathcal F^{[p]}(\log)\supseteq \{0\}
  \end{equation}
  on $X^\circ$ with the following properties.
  \begin{enumerate}
  \item\label{il:_A} The filtrations \eqref{eq:relDiffFilt} and
    \eqref{eq:relDiffFilt2} agree wherever the pair $(X^\circ, \lfloor D^\circ
    \rfloor)$ is snc, and $\phi$ is an snc morphism of $(X^\circ, \lfloor
    D^\circ \rfloor)$.

  \item\label{il:_B} For any $r$, the sheaf $\mathcal F^{[r]}(\log)$ is reflexive,
    and $\mathcal F^{[r+1]}(\log)$ is a saturated subsheaf of $\mathcal F^{[r]}(\log)$.

  \item\label{seq:quotsF} For any $r$, there exists a sequence of sheaves of
    $\sO_{X^\circ}$-modules,
    $$
    0 \to \mathcal F^{[r+1]}(\log) \to \mathcal F^{[r]}(\log) \to
    \phi^*\Omega_{T^{\circ}}^r \otimes \Omega_{X^\circ/T^\circ}^{[p-r]}(\log
    \lfloor D^\circ \rfloor) \to 0,
    $$
    which is exact and analytically locally split in codimension $2$.

  \item\label{lastisom} There exists an isomorphism $\mathcal F^{[p]}(\log) \simeq
    \phi^*\Omega^p_{T^\circ}$.
  \end{enumerate}
\end{mythm}

\begin{rem}[Notation used in Theorem~\ref{thm:relativedifferentialfiltration}]
  If $S$ is any complex variety, we call a sequence of sheaf morphisms,
  \begin{equation}\label{eq:xpl}
    0 \to \sA \to \mathcal B \to \mathcal C \to 0,    
  \end{equation}
  ``exact and analytically locally split in codimension $2$'' if there exists
  a closed subvariety $C \subset S$ of codimension $\codim_S C \geq 3$ and a
  covering of $S \setminus C$ by subsets $(U_i)_{i \in I}$ which are open in
  the analytic topology, such that the restriction of~\eqref{eq:xpl} to $S
  \setminus C$ is exact, and such that the restriction of~\eqref{eq:xpl} to
  any of the open sets $U_i$ splits.  We refer to Footnote~\ref{foot:dlt} on
  Page~\pageref{foot:dlt} for references concerning the notion of a ``dlt
  pair''.
\end{rem}

\begin{proof}[Idea of proof of Theorem~\ref{thm:relativedifferentialfiltration}]
  We give only a very rough and incomplete idea of the proof of
  Theorem~\ref{thm:relativedifferentialfiltration}.  To construct the
  filtration in~\eqref{eq:relDiffFilt2}, one takes the
  filtration~\eqref{eq:relDiffFilt} which exists on the open set $X \setminus
  X_{\rm sing}$ wherever the morphism $\phi$ is snc, and extends the sheaves
  to reflexive sheaves that are defined on all of $X$. It is then not very
  difficult to show that the sequences
  (\ref{thm:relativedifferentialfiltration}.\ref{seq:quotsF}) are exact and
  locally split away from a subset $Z \subset X$ of codimension $\codim_X Z
  \geq 2$. The main point of Theorem~\ref{thm:relativedifferentialfiltration}
  is, however, that it suffices to remove from $X$ a set of codimension
  $\codim_X Z \geq 3$. For this, a careful analysis of the codimension-two
  structure of dlt pairs, cf.~\cite[Sect.~9]{GKKP10}, proves to be key.
\end{proof}

\subsubsection{Residue sequences for reflexive differential forms}

A very important feature of logarithmic differentials is the existence of a
residue map.  In its simplest form consider a smooth hypersurface $D \subset
X$ in a manifold $X$. The residue map is then the cokernel map in the exact
sequence
$$
0 \to \Omega^1_X \to \Omega^1_X(\log D) \to \mathcal O_D \to 0.
$$
More generally, consider a reduced snc pair $(X, D)$. Let $D_0 \subseteq D$ be
any irreducible component and recall from \cite[2.3(b)]{EV92} that there
exists a residue sequence,
$$
\xymatrix{0 \to \Omega^p_X(\log (D - D_0)) \ar[r] &
  \Omega^p_X(\log D) \ar[r]^(.4){\rho^p} &
  \Omega^{p-1}_{D_0}(\log D_0^c) \to 0,
}
$$
where $D_0^c := (D-D_0)|_{D_0}$ denotes the ``restricted complement'' of
$D_0$.  More generally, if $\phi: X \to T$ is an snc morphism of $(X,D)$ we
have a relative residue sequence
\begin{equation}\label{eq:stdResidue}
  \xymatrix{0 \to \Omega^p_{X/T}(\log(D - D_0)) \ar[r] &
    \Omega^p_{X/T}(\log D) \ar[r]^(.4){\rho^p} &
    \Omega^{p-1}_{D_0/T}(\log D_0^c) \to 0.
  }
\end{equation}
The sequence~\eqref{eq:stdResidue} is not a sequence of locally free sheaves
on $X$, and its restriction to $D_0$ will never be exact on the left.
However, an elementary argument, cf.~\cite[Lem.~2.13.2]{KK08}, shows that
restriction of~\eqref{eq:stdResidue} to $D_0$ induces the following exact
sequence
\begin{equation}\label{eq:restrictedsncresidue}
  0 \to \Omega_{D_0/T}^p(\log D_0^c) \xrightarrow{i^p} \Omega^p_{X/T}(\log
  D)|_{D_0} \xrightarrow{\rho^p_D} \Omega_{D_0/T}^{p-1}(\log D_0^c) \to 0,
\end{equation}
which is very useful for inductive purposes.  We recall without proof the
following elementary fact about the residue sequence.

\begin{fact}[Residue map as a test for logarithmic poles]\label{fact:poletest}
  If $\sigma \in H^0\bigl( X,\, \Omega^{p}_{X/T}(\log D) \bigr)$ is any
  reflexive form, then $\sigma \in H^0\bigl( X,\, \Omega^{p}_{X/T}(\log
  (D-D_0)) \bigr)$ if and only if $\rho^{p}(\sigma) = 0$.
\end{fact}

If the pair $(X, D)$ is not snc, no residue map exists in general.  However,
if $(X, D)$ is dlt, then \cite[Cor.~5.52]{KM98} applies to show that $D_0$ is
normal, and an analogue of the residue map $\rho^p$ exists for sheaves of
reflexive differentials.
\begin{figure}
  \centering
  
  \begin{picture}(8,6)(0,0)
    \put( 0.0, 0.0){\includegraphics[height=6cm]{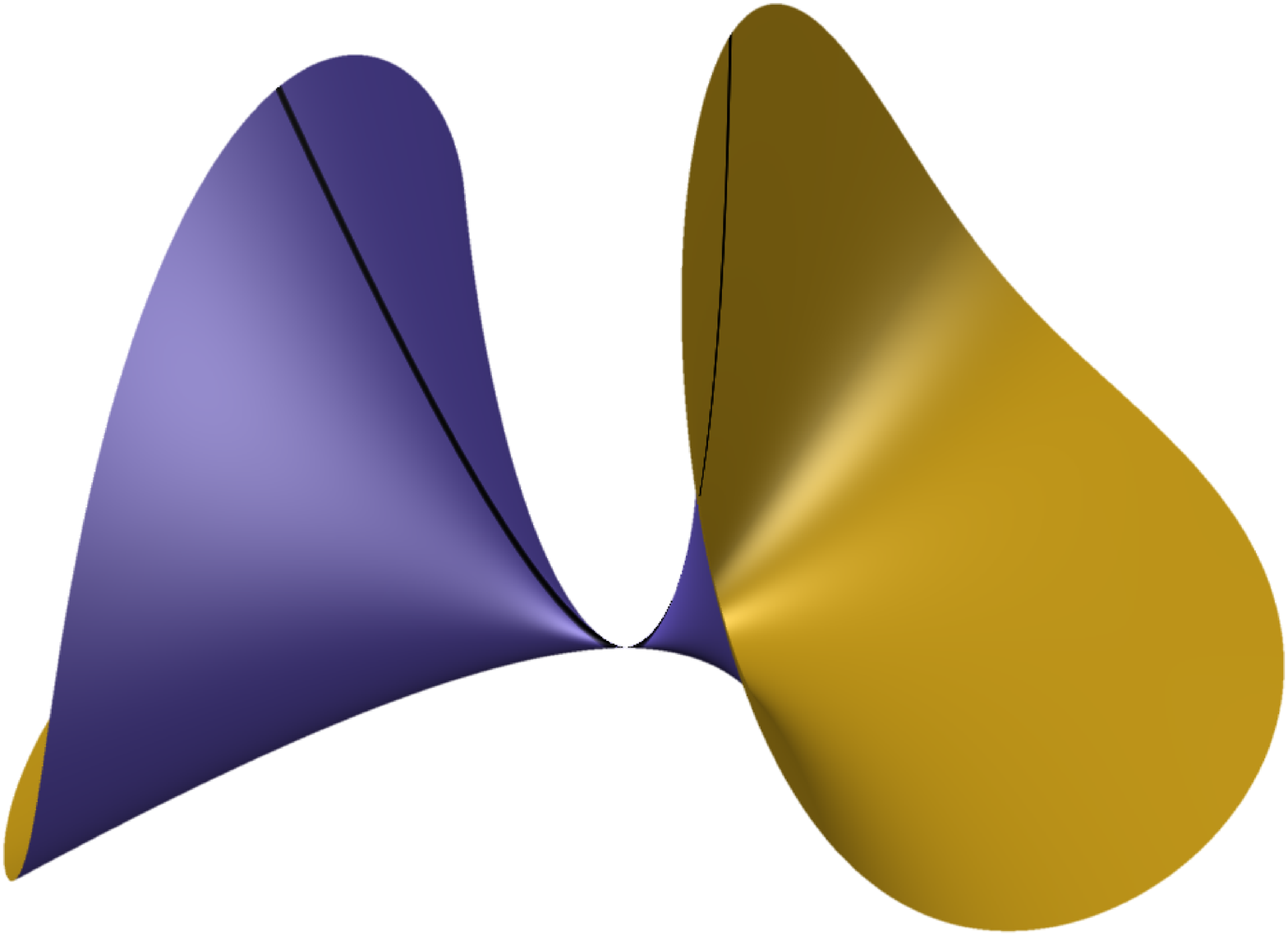}}
    \put(-1.65, 5.05){smooth curve $D$}
    \put( 0.8, 5.0){\vector(2, -1){1.3}}
    \put( 6.5, 4.7){$X$ (singular surface)}
  \end{picture}
  
  \caption{A setup for the residue map on singular spaces.}
  \label{fig:sct}
\end{figure}
To illustrate the problem we are dealing with, consider a normal space $X$
that contains a smooth Weil divisor $D = D_0$, similar to the one sketched in
Figure~\ref{fig:sct}. One can easily construct examples where the singular set
$Z := X_{\rm sing}$ is contained in $D$ and has codimension $2$ in $X$, but
codimension one in $D$. In this setting, a reflexive form $\sigma \in H^0
\bigl( D_0,\, \Omega^{[p]}_{X}(\log D_0) |_{D_0} \bigr)$ is simply the
restriction of a logarithmic form defined outside of $Z$, and the form
$\rho^{[p]}(\sigma)$ is the extension of the well-defined form
$\rho^{p}(\sigma|_{D_0 \setminus Z})$ over $Z$, as a rational form with poles
along $Z \subset D_0$. If the singularities of $X$ are bad, it will generally
happen that the extension $\rho^{[p]}(\sigma)$ has poles of arbitrarily high
order.  Theorem~\ref{thm:relativereflexiveresidue} asserts that this does not
happen when $(X,D)$ is dlt.

\begin{mythm}[\protect{Residue sequences for dlt pairs, \cite[Thm.~11.7]{GKKP10}}]\label{thm:relativereflexiveresidue}
  Let $(X, D)$ be a dlt pair with $\lfloor D \rfloor \neq \emptyset$ and let
  $D_0 \subseteq \lfloor D \rfloor$ be an irreducible component. Let $\phi: X
  \to T$ be a surjective morphism to a normal variety $T$ such that the
  restricted map $\phi|_{D_0}: D_0 \to T$ is still surjective. Then, there
  exists a non-empty open subset $T^\circ \subseteq T$, such that the
  following holds if we denote the preimages as $X^\circ = \phi^{-1}(T^
  \circ)$, $D^\circ = D \cap X^\circ $, and the ``complement'' of $D_0^\circ$
  as $D_0^{\circ,c} := \bigl( \lfloor D^\circ \rfloor - D_0^\circ \bigr)
  |_{D_0^\circ}$.
  \begin{enumerate}
  \item\label{il:RelReflResidue} There exists a sequence
    \begin{multline*}
      \quad \quad \quad \quad 0 \to \Omega^{[r]}_{X^\circ/T^\circ}(\log
      (\lfloor D^\circ\rfloor - D_0^\circ )) \to
      \Omega_{X^\circ/T^\circ}^{[r]}( \log \lfloor D^\circ
      \rfloor ) \\
      \xrightarrow{\rho^{[r]}} \Omega^{[r-1]}_{D^\circ_0/T^\circ}(\log
      D_0^{\circ,c} ) \to 0
    \end{multline*}
    which is exact in $X^\circ$ outside a set of codimension at least $3$. This
    sequence coincides with the usual residue sequence \eqref{eq:stdResidue} wherever
    the pair $(X^\circ, D^\circ)$ is snc and the map $\phi^\circ: X^\circ \to
    T^\circ$ is an snc morphism of $(X^\circ, D^\circ)$.

  \item\label{il:RestrRelReflResidue} The restriction of
    Sequence~(\ref{thm:relativereflexiveresidue}.\ref{il:RelReflResidue}) to
    $D_0$ induces a sequence
    \begin{multline*}
      \quad \quad \quad \quad 0 \to \Omega^{[r]}_{D^\circ_0/T^\circ}(\log
      D_0^{\circ,c} ) \to \Omega_{X^\circ/T^\circ}^{[r]}(\log \lfloor D^\circ
      \rfloor )|_{D_0^\circ} ^{**} \\ %
      \xrightarrow{\rho^{[r]}_{D^\circ_0}}
      \Omega^{[r-1]}_{D^\circ_0/T^\circ}(\log D_0^{\circ,c} ) \to 0
    \end{multline*}
    which is exact on $D_0^\circ$ outside a set of codimension at least $2$ and
    coincides with the usual restricted residue sequence
    \eqref{eq:restrictedsncresidue} wherever the pair $(X^\circ, D^\circ)$ is snc and
    the map $\phi^\circ: X^\circ \to T^\circ$ is an snc morphism of $(X^\circ,
    D^\circ)$. \qed
  \end{enumerate}
\end{mythm}

As before, the proof of Theorem~\ref{thm:relativereflexiveresidue} relies on
our knowledge of the codimension-two structure of dlt pairs.
Fact~\ref{fact:poletest} and Theorem~\ref{thm:relativereflexiveresidue}
together immediately imply that the residue map for reflexive differentials
can be used to check if a reflexive form has logarithmic poles along a given
boundary divisor.

\begin{rem}[Residue map as a test for logarithmic poles]\label{rem:poletest}
  In the setting of Theorem~\ref{thm:relativereflexiveresidue}, if $\sigma \in
  H^0\bigl( X,\, \Omega^{[p]}_X(\log \lfloor D \rfloor) \bigr)$ is any
  reflexive form, then $\sigma \in H^0\bigl( X,\, \Omega^{[p]}_X(\log \lfloor
  D \rfloor -D_0) \bigr)$ if and only if $\rho^{[p]}(\sigma) = 0$.
\end{rem}

\subsubsection{The residue map for $1$-forms}
\label{sec:part2-last}

Let $X$ be a smooth variety and $D \subset X$ a smooth, irreducible
divisor. The first residue sequence~\eqref{eq:stdResidue} of the pair $(X,D)$
then reads
$$
0 \to \Omega^1_D \to \Omega^1_X(\log D)|_D \xrightarrow{\rho^1} \sO_D \to 0,
$$
and we obtain a connecting morphism of the long exact cohomology sequence,
$$
\delta : H^0 \bigl( D, \sO_D \bigr) \to H^1 \bigl( D, \Omega^1_D \bigr).
$$
In this setting, the standard description of the first Chern class in terms of
the connecting morphism, \cite[III.~Ex.~7.4]{Ha77}, asserts that
\begin{equation}\label{eq:c1descr}
  c_1\bigl( \sO_X(D)|_D \bigr) =
  \delta(\mathbf{1}_D) \in H^1 \bigl( D, \Omega^1_D \bigr),
\end{equation}
where $\mathbf{1}_D$ is the constant function on $D$ with value
one. Theorem~\ref{thm:Chernclass} generalises Identity~\eqref{eq:c1descr} to
the case where $(X,D)$ is a reduced dlt pair with irreducible boundary
divisor.

\begin{mythm}[\protect{Description of Chern classes, \cite[Thm.~12.2]{GKKP10}}]\label{thm:Chernclass}
  Let $(X, D)$ be a dlt pair, $D = \lfloor D \rfloor$ irreducible.  Then,
  there exists a closed subset $Z \subset X$ with $\codim_{X}Z \geq 3$ and a
  number $m \in \mathbb{N}$ such that $mD$ is Cartier on $X^\circ := X
  \setminus Z$, such that $D^\circ := D \cap X^\circ$ is smooth, and such that
  the restricted residue sequence
  \begin{equation}\label{eq:simplerestrictedreflexiveresidue}
    0 \to \Omega_D^1 \to \Omega_X^{[1]}(\log D)|_D^{**}
    \overset{\rho_D}{\longrightarrow} \sO_{D} \to 0
  \end{equation}
  defined in Theorem~\ref{thm:relativereflexiveresidue} is exact on
  $D^\circ$. Moreover, for the connecting homomorphism $\delta$ in the
  associated long exact cohomology sequence
  $$
  \delta : H^0 \bigl(D^\circ,\, \sO_{D^\circ} \bigr) \to
  H^1\bigl(D^\circ,\, \Omega_{D^\circ}^1\bigr)
  $$
  we have
  \begin{equation}
    \delta(m\cdot \mathbf{1}_{D^\circ} ) = 
    c_1 \left(\left. \sO_{X^\circ}(mD^\circ)\right|_{D^\circ}\right).
  \end{equation}
\end{mythm}

\subsection{Existence of a pull-back morphism, idea of proof}
\label{ssec:IOPpb}

The proof of Theorem~\ref{thm:generalpullback} is rather involved. To
illustrate the idea of the proof, we concentrate on a very special case, and
give only indications what needs to be done to handle the general setup.

\subsubsection{Simplifying assumptions and setup of notation}

The following simplifying assumptions will be maintained throughout the
present Section~\ref{ssec:IOPpb}.

\begin{sass}\label{sass:steenbrink}
  The space $X$ has dimension $n := \dim X \geq 3$. It is klt, has only one
  single isolated singularity $x \in X$, and the divisor $D$ is empty. The
  morphism $\gamma: Z \to X$ is a resolution of singularities, whose
  exceptional set $E \subset Z$ is a divisor with simple normal crossing
  support.
\end{sass}

To prove Theorem~\ref{thm:generalpullback}, we need to show in essence that
reflexive differential forms $\sigma \in H^0 \bigl( X,\, \Omega^{[p]}_X
\bigr)$ pull back to give differential forms $\widetilde{\sigma} \in H^0
\bigl( Z,\, \Omega^p_Z \bigr)$. The following observation, an immediate
consequence of Fact~\ref{fact:extension1}, turns out to be key.

\begin{obs}\label{obs:xxL}
  To give a reflexive differential $\sigma \in H^0 \bigl( X,\, \Omega^{[p]}_X
  \bigr)$, it is equivalent to give a differential form $\sigma^\circ \in H^0
  \bigl( X \setminus X_{\rm sing},\, \Omega^p_X \bigr)$, defined on the smooth
  locus of $X$. Since the resolution map identifies the open subvarieties $Z
  \setminus E$ and $X \setminus X_{\rm sing}$, we see that to give a reflexive
  differential $\sigma \in H^0 \bigl( X,\, \Omega^{[p]}_X \bigr)$, it is in
  fact equivalent to give a differential form $\widetilde{\sigma}^\circ \in
  H^0 \bigl( Z \setminus E,\, \Omega^p_Z \bigr)$.
\end{obs}

In essence, Observation~\ref{obs:xxL} says that to show
Theorem~\ref{thm:generalpullback}, we need to prove that the natural
restriction map
\begin{equation}\label{eq:restr1} 
  H^0 \bigl( Z,\, \Omega^p_Z \bigr) \to H^0 \bigl( Z \setminus E,\, 
  \Omega^p_Z \bigr)
\end{equation}
is in fact surjective. In other words, we need to show that any differential
form on $Z$, which is defined outside of the $\gamma$-exceptional set $E$,
automatically extends across $E$, to give a differential form defined on all
of $Z$. This is done in two steps. We first show that the restriction map
\begin{equation}\label{eq:restr2}
  H^0 \bigl( Z,\, \Omega^p_Z(\log E) \bigr)
  \to H^0 \bigl( Z \setminus E,\, \Omega^p_Z(\log E) \bigr) = H^0 \bigl( Z
  \setminus E,\, \Omega^p_Z \bigr)
\end{equation}
is surjective. In other words, we show that any differential form on $Z$,
defined outside of $E$, extends as a form with logarithmic poles along
$E$. Secondly, we show that the natural inclusion map
\begin{equation}\label{eq:restr3} 
  H^0 \bigl( Z,\, \Omega^p_Z \bigr) \to H^0
  \bigl( Z,\, \Omega^p_Z(\log E) \bigr)
\end{equation}
is likewise surjective. In other words, we show that globally defined
differentials forms on $Z$, which are allowed to have logarithmic poles along
$E$, really do not have any poles. Surjectivity of the morphisms
\eqref{eq:restr2} and \eqref{eq:restr3} together will then imply surjectivity
of \eqref{eq:restr1}, finishing the proof of
Theorem~\ref{thm:generalpullback}.

The arguments used to prove surjectivity of \eqref{eq:restr2} and
\eqref{eq:restr3}, respectively, are of rather different nature. We will
sketch the arguments in Sections~\ref{sec:ext2} and \ref{sec:ext3} below.

\subsubsection{Surjectivity of the restriction map \eqref{eq:restr2}}
\label{sec:ext2}

Under the Simplifying Assumptions~\ref{sass:steenbrink}, surjectivity of the
map \eqref{eq:restr2} has essentially been shown by Steenbrink and van
Straten, \cite{SS85}. We give a brief synopsis of their line of argumentation
and indicate additional steps of argumentation required to handle the general
setting. To start, recall from \cite[III~ex.~2.3e]{Ha77} that the map
\eqref{eq:restr2} is part of the standard sequence that defines cohomology
with supports,
\begin{multline}\label{eq:fcws}
  \cdots \to H^0\bigl(Z,\, \Omega^p_Z(\log E) \bigr) \to H^0\bigl(Z \setminus
  E,\, \Omega^p_Z(\log E) \bigr) \to\\ \to H^1_E\bigl( Z,\, \Omega^p_Z(\log
  E)\bigr) \to \cdots
\end{multline}
We aim to show that the last term in~\eqref{eq:fcws} vanishes. There are two
main ingredients to the proof: formal duality and Steenbrink's vanishing
theorem.

\begin{mythm}[\protect{Formal duality theorem for cohomology with support, \cite[Chapt.~3, Thm.~3.3]{Hartshorne1970}}]\label{thm:duality}
  Under the Assumptions~\ref{sass:steenbrink}, if $\mathcal F$ is any locally
  free sheaf on $Z$ and $0 \leq j\leq \dim Z$ any number, then there
  exists an isomorphism
  $$
  \left((R^j\gamma_* \mathcal F)_x \right)^{\widehat\ } \cong H^{n-j}_E
  \bigl(Z,\, \mathcal F^*\otimes \omega_Z \bigr)^*,
  $$
  where $\,\widehat{~}\,$ denotes completion with respect to the maximal ideal
  $\mathfrak{m}_x$ of the point $x \in X$, and where $ n = \dim X = \dim
  Z$. \qed
\end{mythm}

A brief introduction to formal duality, together with a readable,
self-contained proof of Theorem~\ref{thm:duality} is found in
\cite[Appendix~A]{GKK08} while Hartshorne's lecture notes
\cite{Hartshorne1970} are the standard reference for these matters.

\begin{mythm}[\protect{Steenbrink vanishing, \cite[Thm.~2.b]{Steenbrink85}}]
  If $p$, $q$ are any two numbers with $p+q > \dim Z$, then $R^q \gamma_*
  \bigl(\mathcal J_E \otimes \Omega^p_Z(\log E) \bigr) = 0$. \qed
\end{mythm}

\begin{rem}
  Steenbrink's vanishing theorem is proven using local Hodge theory of
  isolated singularities. For $p = n$, the sheaves $\Omega^n_Z$ and $\mathcal
  J_E \otimes \Omega^n_Z(\log E)$ are isomorphic. In this case, the Steenbrink
  vanishing theorem reduces to Grauert-Riemenschneider vanishing, \cite{GR70}.
\end{rem}

Setting $\mathcal F := \mathcal J_E \otimes \Omega^{n-p}_Z(\log E)$ and using
that $\mathcal F^* \otimes \omega_Z \cong \Omega^p_Z(\log E)$, formal duality
and Steenbrink vanishing together show that $H^1_E \bigl( Z,\, \Omega^p_Z(\log
E )\bigr) = 0$, for $p < \dim Z-1$, proving surjectivity of \eqref{eq:restr2}
for these values of $p$. The other cases need to be treated separately.
\begin{description}
\item[case $\pmb{p=n}$] After passing to an index-one cover, surjectivity of
  \eqref{eq:restr2} in case $p=n$ follows almost directly from the definition
  of klt, cf.~\cite[Sect.~5]{GKK08}.
\item[case $\pmb{p=n-1}$] In this case one uses the duality between
  $\Omega^{n-1}_Z$ and the tangent sheaf $T_Z$, and the fact that any section
  in the tangent sheaf of $X$ always lifts to the canonical resolution of
  singularities, cf.~\cite[Sect.~6]{GKK08}.
\end{description}

\subsubsection*{General case}

The argument outlined above, using formal duality and Steenbrink vanishing,
works only because we were assuming that the singularities of $X$ are
isolated. In the general case, where the Simplifying
Assumptions~\ref{sass:steenbrink} do not necessarily hold, this is not
necessarily the case. In order to deal with non-isolated singularities, one
applies a somewhat involved cutting-down procedure, as indicated in
Figure~\ref{fig:TSWAH}.
\begin{figure}
  \centering

  \ \\

  $$
  \xymatrix{
    \begin{picture}(4,4)(0,0)
      \put( 0.0, 4.2){$Z$, resolution of singularities}
      \put( 0.0, 0.2){\includegraphics[height=3.5cm]{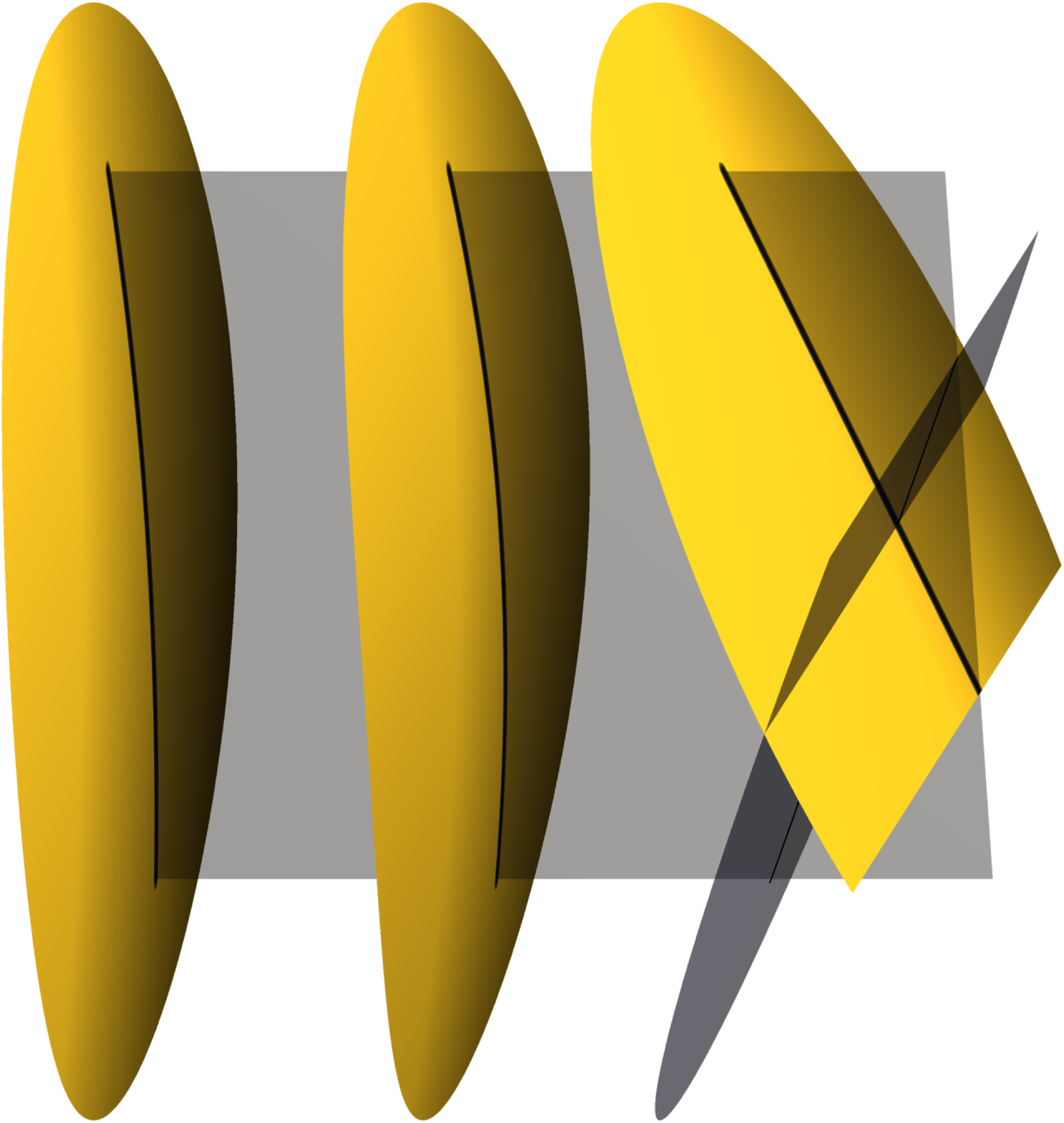}}
      \put( 3.4, 3.7){\scriptsize divisor $E_1$}
      \put( 3.4, 3.6){\vector(-1, -1){0.4}}
      \put( 2.9, 0.3){\scriptsize divisor $E_0$}
      \put( 2.8, 0.4){\vector(-2, 1){0.4}}
    \end{picture}
    \ar[rr]^{\gamma}_{\text{resolution map}} &&
    \begin{picture}(4,4)(0,0)
      \put( 0.0, 4.2){singular space $X$}
      \put( 0.0, 0.2){\includegraphics[height=3.5cm]{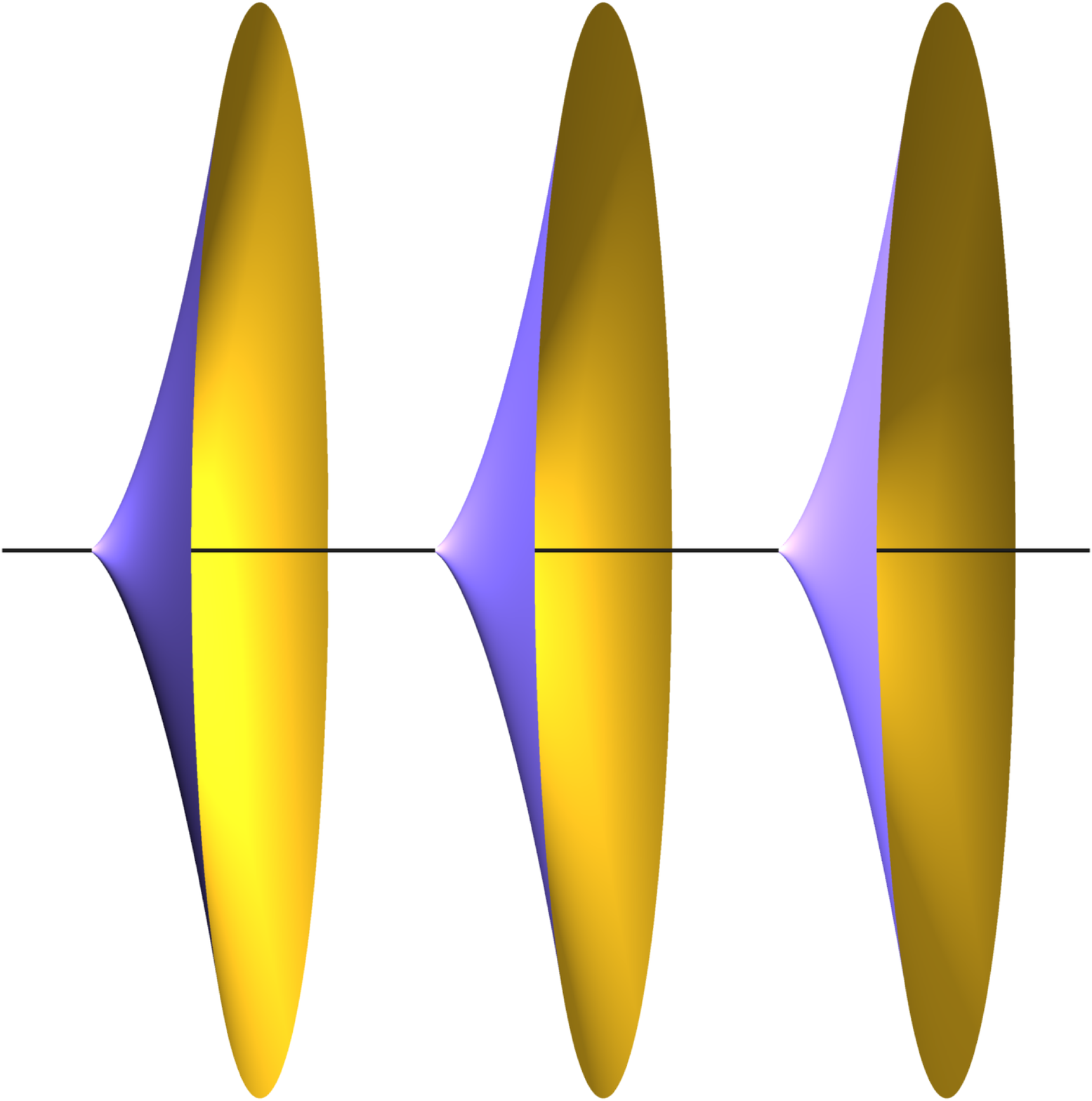}}
      \put( 3.9, 2.6){\scriptsize point $\gamma(E_0)$}
      \put( 3.8, 2.6){\vector(-2, -1){1.2}}
      \put( 2.4, 1.87){\scriptsize $\bullet$}
      \put( 3.8, 1.3){\scriptsize curve $\gamma(E_1)$}
      \put( 3.8, 1.5){\vector(-1, 1){0.4}}
    \end{picture}
  }
  $$

  {\small The figure sketches a situation where $X$ is a threefold whose
    singular locus is a curve. Near the general point of the singular locus,
    the variety $X$ looks like a family of isolated surfaces
    singularities. The exceptional set $E$ of the resolution map $\gamma$
    contains two irreducible divisors $E_0$ and $E_1$.}

  \caption{Non-isolated singularities}
  \label{fig:TSWAH}
\end{figure}
This way, it is often possible to view non-isolated log canonical
singularities a family of isolated singularities, where surjectivity of
\eqref{eq:restr2} can be shown on each member of the family. To conclude that
it holds on all of $Z$, the following strengthening of Steenbrink vanishing is
required.

\begin{mythm}[\protect{Steenbrink-type vanishing for log canonical pairs, \cite[Thm.~14.1]{GKKP10}}]\label{thm:Omegavanishing}
  Let $(X, D)$ be a log canonical pair of dimension $n \geq 2$.  If $\gamma :
  Z \to X$ is a log resolution of singularities with exceptional set $E$ and
  $$
  \Delta := \supp \bigl(E + \gamma^{-1} \lfloor D\rfloor \bigr),
  $$
  then $R^{n-1}\gamma_*\bigl(\Omega^{p}_Z(\log \Delta) \otimes \mathcal O_Z
  (-\Delta)\bigr) = 0$ for all $0\leq p \leq n$. \qed
\end{mythm}

The proof of Theorem~\ref{thm:Omegavanishing} essentially relies on the fact
that log canonical pairs are Du~Bois, \cite{KKLogCanonicalDuBois}. The Du~Bois
property generalises the notion of rational singularities. For an overview,
see \cite{Kovacs-Schwede09}.

\subsubsection{Surjectivity of the inclusion map \eqref{eq:restr3}}
\label{sec:ext3}

Let $\sigma \in H^0 \bigl( Z,\, \Omega^p_Z (\log E) \bigr)$ be any differential
form on $Z$ that is allowed to have logarithmic poles along $E$. To show
surjectivity of the inclusion map \eqref{eq:restr3}, we need to show that
$\sigma$ really does not have any poles along $E$. To give an idea of the
methods used to prove this, we consider only the case where $p > 1$. We
discuss two particularly simple cases first.

\subsubsection*{The case where $E$ is irreducible}

Assume that $E$ is irreducible. To show that $\sigma$ does not have any
logarithmic poles along $E$, recall from Fact~\ref{fact:poletest} that it
suffices to show that $\sigma$ is in the kernel of the residue map
$$
\rho^p : H^0\bigl( Z,\, \Omega^p_Z(\log E) \bigr) \to H^0\bigl( E,\,
\Omega^{p-1}_E \bigr).
$$
On the other hand, we know from a result of Hacon-McKernan,
\cite[Cor.~1.5(2)]{HMcK07}, that $E$ is rationally connected, so that
$H^0\bigl( E,\, \Omega^{p-1}_E \bigr) = 0$. This clearly shows that $\sigma$
is in the kernel of $\rho^p$ and completes the proof when $E$ is irreducible.

\subsubsection*{The case where $(Z,E)$ admits a simple minimal model program}

In general, the divisor $E$ need not be irreducible. Let us therefore consider
the next difficult case where $E$ is reducible with two components, say $E =
E_1 \cup E_2$. The resolution map $\gamma$ will then factor via a
$\gamma$-relative minimal model program of the pair $(Z, E)$, which we assume
for simplicity to have the following particularly special form,
sketched\footnote{The computer code used to generate the images in
  Figure~\ref{fig:srklt} is partially taken from \cite{Baum}.} also in
Figure~\ref{fig:srklt}.
\begin{figure}
  \centering

  \ \\

  $$
  \xymatrix{
    \begin{picture}(4.8,4)(0,0)
      \put( 1.0, 0.2){\includegraphics[height=3.5cm]{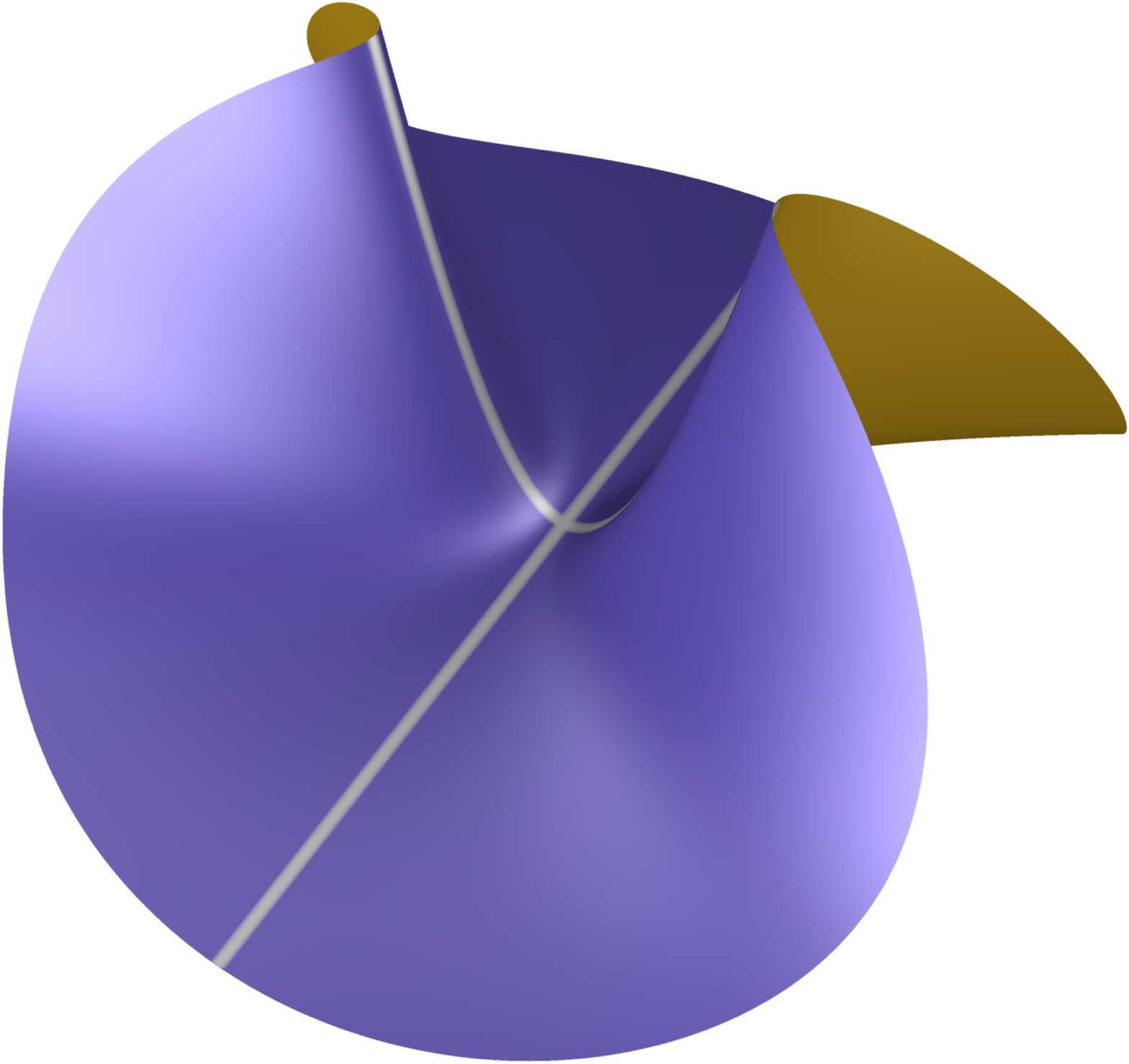}}
      \put( 0.0, 4.2){snc surface pair $(Z, E_1+E_2)$}
      \put( 0.2, 3.4){\scriptsize divisor $E_1$}
      \put( 1.5, 3.4){\vector(4, -1){0.6}}
      \put(-0.5, 1.2){\scriptsize divisor $E_2$}
      \put( 0.8, 1.2){\vector(4, -1){1.1}}
    \end{picture} \quad
    \ar[rr]^(.54){\lambda_1}_(.54){\text{contracts $E_1$}}
    \ar@<-10mm>@/_5mm/[drr]_(.55){\text{resolution map\ }}^(.55){\gamma} &&
    \begin{picture}(4,4)(0,0)
      \put( 0.0, 0.2){\includegraphics[height=3.5cm]{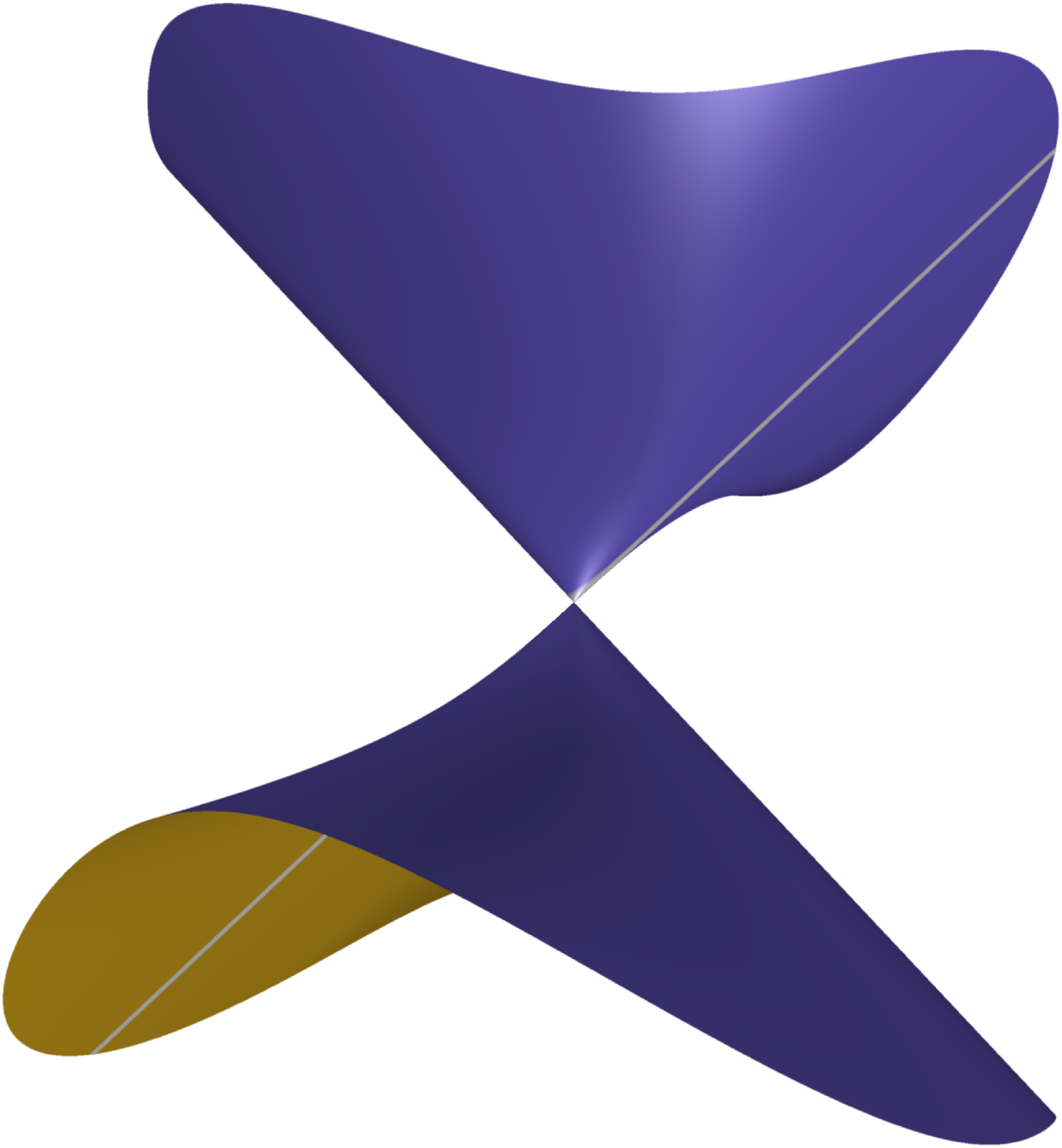}}
      \put( 0.0, 4.2){dlt surface pair $(Z_1, E_{2,1})$}
      \put( 2.8, 1.4){\scriptsize divisor $E_{2,1}$}
      \put( 2.7, 1.6){\vector(-1, 4){0.2}}
    \end{picture}
    \ar[d]_(.55){\lambda_2}^(.55){\text{contracts $E_{2,1}$}} \\ &&
    \begin{picture}(3.5, 3.2)(0,0)
      \put( 0.0, 0.0){\includegraphics[width=3.5cm]{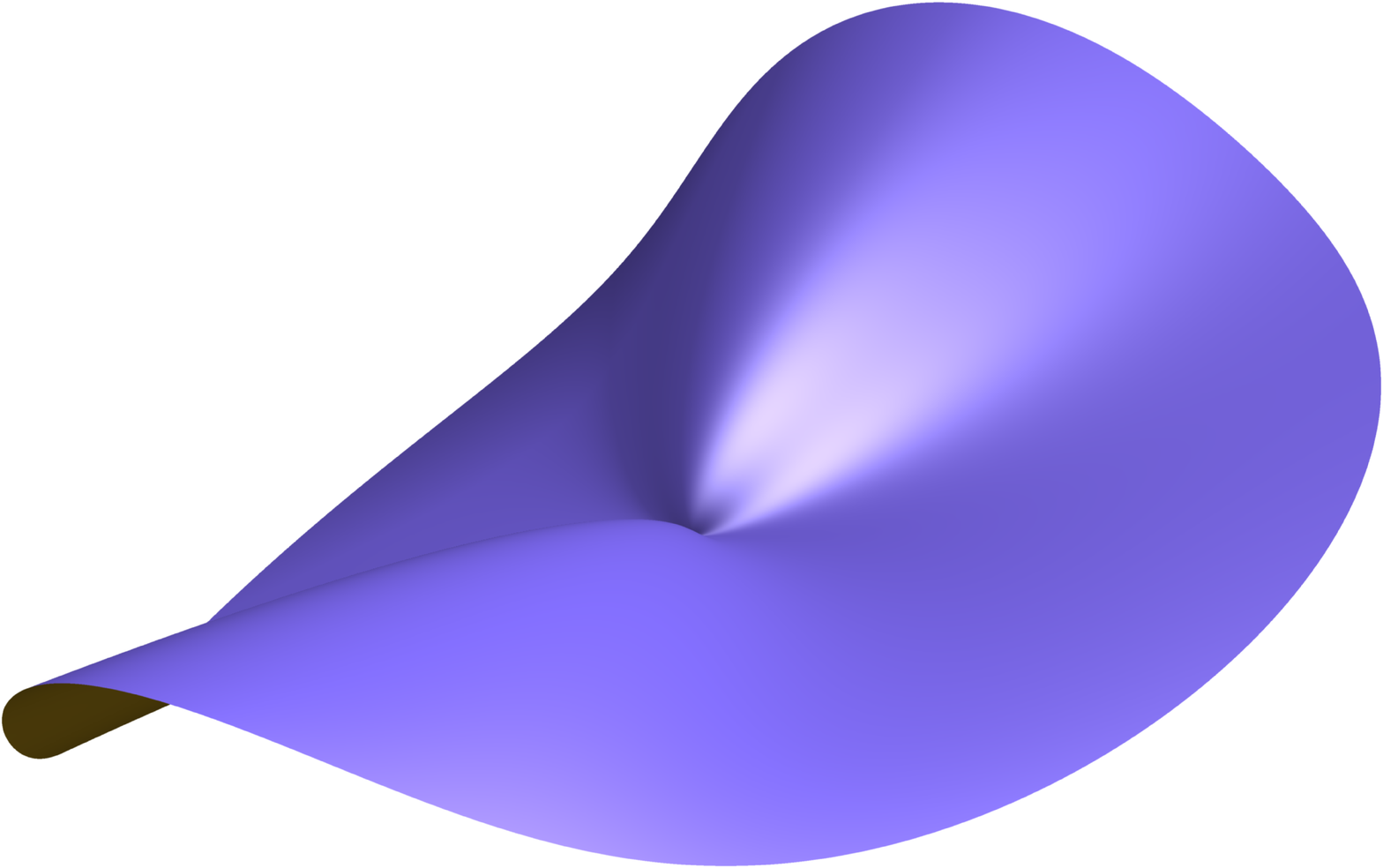}}
      \put( 0.0, 2.5){klt surface $X$}
    \end{picture}
  }
  $$

  {\small This sketch shows a resolution of an isolated klt surface singularity, and the
    decomposition of the resolution map given by the minimal model program of the snc pair
    $(Z, E_1+E_2)$.}

  \caption{Resolution of an isolated klt surface singularity}
  \label{fig:srklt}
\end{figure}
$$
\xymatrix{ %
  Z = Z_0 \ar[rrrr]^(.55){\lambda_1}_(.55){\text{contracts $E_1$ to a point}}
  &&&& Z_1 \ar[rrrrr]^{\lambda_2}_{\text{contracts $E_{2,1} :=
      (\lambda_1)_*(E_2)$ to a point}} &&&&& X. }
$$
In this setting, the arguments outlined above apply verbatim to show that
$\sigma$ has no poles along the divisor $E_1$. To show that $\sigma$ does not
have any poles along the remaining component $E_2$, observe that it suffices
to consider the induced reflexive form on the possibly singular space $Z_1$,
say $\sigma_1 \in H^0 \bigl( Z_1 ,\, \Omega^{[p]}_{Z_1} (\log E_{2,1})
\bigr)$, where $E_{2,1} := (\lambda_1)_*(E_2)$, and to show that $\sigma_1$
does not have any poles along $E_{2,1}$.  For that, we follow the same line of
argument once more, accounting for the singularities of the pair $(X_1,
E_{2,1})$.

The pair $(X_1, E_{2,1})$ is dlt, and it follows from adjunction that the
divisor $E_{2,1}$ is necessarily normal, \cite[Cor.~5.52]{KM98}. Using the
residue map for reflexive differentials on dlt pairs that was constructed in
Theorem~\ref{thm:relativereflexiveresidue},
$$
\rho^{[p]} : H^0\bigl( X_1,\, \Omega^{[p]}_{Z_1}(\log E_{2,1}) \bigr) \to
H^0\bigl( E_{2,1},\, \Omega^{[p-1]}_{E_{2,1}} \bigr),
$$
we have seen in Remark~\ref{rem:poletest} that it suffices to show that
$\rho^{[p]}(\sigma_1) = 0$. Because the morphism $\lambda_2$ contracts the
divisor $E_{2,1}$ to a point, the result of Hacon-McKernan will again apply to
show that $E_{2,1}$ is rationally connected. Even though there are numerous
examples of rationally connected spaces that carry non-trivial reflexive
forms, we claim that in our special setup we do have the vanishing
\begin{equation}\label{eq:nrdf}
  H^0\bigl( E_{2,1},\, \Omega^{[p-1]}_{E_{2,1}} \bigr) = 0.
\end{equation}

For this, recall from the adjunction theory for Weil divisors on normal
spaces, \cite[Chapt.~16 and Prop.~16.5]{FandA92} and \cite[Sect.~3.9 and
Glossary]{MR2352762}, that there exists a Weil divisor $D_E$ on the normal
variety $E_{2,1}$ which makes the pair $(E_{2,1}, D_E)$ klt. Now, if we knew
that the extension theorem would hold for the pair $(E_{2,1}, D_E)$, we can
prove the vanishing statement~\eqref{eq:nrdf}, arguing exactly as in the proof
of Corollary~\ref{cor:kltRCC}, where we show the non-existence of reflexive
forms on rationally connected klt spaces as a corollary of the Pull-Back
Theorem~\ref{thm:generalpullback}. Since $\dim E_{2,1} < \dim X$, this
suggests an inductive proof, beginning with easy-to-prove extension theorems
for reflexive forms on surfaces, and working our way up to higher-dimensional
varieties. The proof in \cite{GKKP10} follows this inductive pattern.

\subsubsection*{The general case}

To handle the general case, where the Simplifying
Assumptions~\ref{sass:steenbrink} do not necessarily hold true, we need to
work with pairs $(X,D)$ where $D$ is not necessarily empty, the
$\gamma$-relative minimal model program might involve flips, and the
singularities of $X$ need not be isolated. All this leads to a slightly
protracted inductive argument, heavily relying on cutting-down methods and
outlined in detail in \cite[Sect.~19]{GKKP10}.

\subsection{Open problems}

In view of the Viehweg-Zuo construction, it would be very interesting to know
if a variant of the Pull-Back Theorem~\ref{thm:generalpullback} holds for
symmetric powers of $\Omega^{[1]}_X(\log D)$, or for other tensor powers. As
shown by examples, cf.~\cite[Ex.~3.1.3]{GKK08}, the na\"ive generalisation of
Theorem~\ref{thm:generalpullback} is wrong. Still, it seems conceivable that a
suitable generalisation, perhaps formulated in terms of Campana's orbifold
differentials, might hold. However, note that several of the key ingredients
used in the proof of Theorem~\ref{thm:generalpullback}, including Steenbrink's
vanishing theorem, rely on (local) Hodge theory, for which no version is known
for tensor powers of differential forms.

\begin{question}
  Is there a formulation of the Pull-Back Theorem~\ref{thm:generalpullback}
  that holds for symmetric and other tensor powers of differential forms?
\end{question}

Examples suggest that the Pull-Back Theorem~\ref{thm:generalpullback} is
optimal, and that the class of log canonical pairs is the natural class of
spaces where a pull-back theorem can hold. 

\begin{question}
  To what extend is the Pull-Back Theorem~\ref{thm:generalpullback} optimal?
  Is there a version of the pull-back theorem that does not require the log
  canonical divisor $K_X+D$ to be $\mathbb Q$-Cartier? If we are interested
  only in special values of $p$, is the divisor $\Delta$ the smallest
  possible?
\end{question}

The last question concerns the generalisation of the Bogomolov-Sommese
vanishing theorem. One of the main difficulties with its current formulation
is the requirement that the sheaf $\mathcal A$ be $\mathbb Q$-Cartier. We have
seen in Section~\ref{ssec:motiv} how interesting reflexive subsheaves
$\mathcal A \subseteq \Omega^{[p]}_X$ can often be constructed using the
Harder-Narasimhan filtration. Unless the space $X$ is $\mathbb Q$-factorial,
there is, however, no way to guarantee that a sheaf constructed this way will
actually be $\mathbb Q$-Cartier. The property to be $\mathbb Q$-factorial,
however, is not stable under taking hyperplane sections and difficult to
guarantee in practise.

\begin{question}
  Is there a version of the generalised Bogomolov-Sommese vanishing theorem,
  Corollary~\ref{cor:BS}, that does not require the sheaf $\mathcal A$ to be
  $\mathbb Q$-Cartier?
\end{question}

\section{Viehweg's conjecture for families over threefolds, sketch of proof}
\label{sec:ideaOfProof}

\subsection{A special case of the Viehweg conjecture}

We conclude this paper by sketching a proof of the Viehweg
Conjecture~\ref{conj:Viehweg} in one special case, illustrating the use of the
methods introduced in Sections~\ref{sec:VZ} and \ref{sec:ext}. As in
Section~\ref{ssec:motiv} we consider a family $f^\circ : X^\circ \to Y^\circ$
of canonically polarised varieties over a quasi-projective threefold. Assuming
that $f^\circ$ is of maximal variation, we would like to show that the
logarithmic Kodaira dimension $\kappa(Y^\circ)$ cannot be zero.

\begin{myprop}[Partial answer to Viehweg's conjecture]\label{prop:vzsimple2}
  Let $f^\circ : X^\circ \to Y^\circ$ be a smooth, projective family of
  canonically polarised varieties over a smooth, quasi-projective base
  manifold of dimension $\dim Y^\circ = 3$. Assume that the family $f^\circ$
  is of maximal variation, i.e., that $\Var(f^\circ) = \dim Y^\circ$. Then
  $\kappa(Y^\circ) \not = 0$.
\end{myprop}

The proof of Proposition~\ref{prop:vzsimple2} follows the line of
argumentation outlined in Section~\ref{ssec:motiv}. We prove that the Picard
number of a suitable minimal model cannot be one, thereby exhibiting a fibre
space structure to which induction can be applied. The presentation follows
\cite[Sect.~9]{KK08c}.

\subsection{Sketch of proof}

In essence, we follow the line of argument sketched in
Section~\ref{ssec:motiv}. We argue by contradiction, i.e., we maintain the
assumptions of Proposition~\ref{prop:vzsimple2} and assume in addition that
$\kappa(Y^\circ) = 0$.

\subsubsection{Setup of notation}

As before, choose a smooth compactification $Y \supseteq Y^\circ$ such that $D
:= Y \setminus Y^\circ$ is a divisor with only simple normal crossings. Let
$\lambda : Y \dasharrow Y_\lambda$ be the rational map obtain by a run of the
minimal model program for the pair $(Y,D)$ and set $D_\lambda :=
\lambda_*(D)$. The following is then known to hold.
\begin{enumerate}
\item The variety $Y_\lambda$ is normal and $\mathbb Q$-factorial.
\item The variety $Y_\lambda$ is log terminal. The pair $(Y_\lambda,
  D_\lambda)$ is dlt.
\item There exists a number $m'$ such that $m' \bigl( K_{Y_\lambda} +
  D_\lambda \bigr) \equiv 0$. In particular, the divisor $K_{Y_\lambda} +
  D_\lambda$ is numerically trivial.
\end{enumerate}
By Viehweg-Zuo's Theorem~\ref{thm:VZ}, there exists a number $m > 0$ and a big
invertible sheaf $\sA \subseteq \Sym^m \Omega^1_Y (\log D)$. As we have seen
in Proposition~\ref{lem:pushdownA}, this induces a reflexive sheaf
$\sA_\lambda \subseteq \Sym^{[m]} \Omega^1_{Y_\lambda} (\log D_\lambda)$ of
rank one and Kodaira-Iitaka dimension $\kappa(\sA_\lambda) = \dim Y_\lambda$.

\subsubsection{The Harder-Narasimhan filtration of $\Omega^{[1]}_{Y_\lambda}(\log D_\lambda)$}

As in Section~\ref{ssec:motiv} above, we employ the Harder-Narasimhan
filtration to obtain additional information about the space $Y_\lambda$.

\begin{claim}\label{claim:5-4}
  The divisor $D_\lambda$ is not empty.
\end{claim}
\begin{proof}
  For simplicity, we prove Claim~\ref{claim:5-4} only in case where the
  canonical divisor $K_{Y_\lambda}$ is Cartier, and where the space
  $Y_\lambda$ therefore has only canonical singularities. For a proof in the
  general setup, the same line of argumentation applies after passing to a
  global index-one cover. We argue by contradiction and assume that $D_\lambda
  = 0$.
  
  As before, let $C \subseteq Y_\lambda$ be a general complete intersection
  curve in the sense of Mehta-Ramanathan, cf.~\cite[Sect.~II.7]{HL97}.  Since
  the general complete intersection curve $C$ avoids the singular locus of
  $Y_\lambda$, we obtain that the restricted sheaf of Kähler differentials
  $\Omega^1_{Y_\lambda}|_C$ as well as its dual $\mathcal T_{Y_\lambda}|_C$,
  the restriction of the tangent sheaf, are locally free.  Further, the
  numerical triviality of $K_{Y_\lambda} \equiv K_{Y_\lambda} + D_\lambda$
  implies that
  $$
  K_{Y_\lambda}.C = c_1 \bigl( \Omega^{[1]}_{Y_\lambda}(\log D_\lambda)
  \bigr).C = c_1\bigl(\Sym^{[m]} \Omega^1_{Y_\lambda}(\log D_\lambda) \bigr).C
  = 0.
  $$
  On the other hand, since $\sA_\lambda$ is big, we have that $c_1(\mathcal
  A_\lambda).C > 0$. As in the proof of Proposition~\ref{prop:VZec1}, this
  implies that the restricted sheaves $\Omega^1_{Y_\lambda}|_C$ as well as its
  dual $\mathcal T_{Y_\lambda}|_C$, are not semistable. The maximal
  destabilising subsheaf of $\mathcal T_{Y_\lambda}|_C$ is semistable and of
  positive degree, hence ample. In this setup, a variant \cite[Cor.~5]{KST07}
  of Miyaoka's uniruledness criterion \cite[Cor.~8.6]{Miy85} applies to give
  the uniruledness of $Y_\lambda$. For more details on this criterion, see the
  survey \cite{KS06}.

  To finish the argument, let $r : W \to Y_\lambda$ be a resolution of
  singularities.  Since uniruledness is a birational property, the space $W$
  is uniruled and therefore has Kodaira-dimension $\kappa(W)=-\infty$. On the
  other hand, since $Y_\lambda$ has only canonical singularities, the $\mathbb
  Q$-linear equivalence class of the canonical bundle $K_W$ is given as
  $$
  K_W \equiv r^*(K_{Y_\lambda}) + (\text{effective, $r$-exceptional divisor}).
  $$
  But because $K_{Y_\lambda}$ is $\mathbb Q$-linearly equivalent to the
  trivial divisor, we obtain that $\kappa(W) \geq 0$, a contradiction.
\end{proof}

\subsubsection{Further contractions}

Claim~\ref{claim:5-4} implies that $K_{Y_\lambda} \equiv - D_\lambda$ and it
follows that for any rational number $0 < \varepsilon < 1$,
\begin{equation}\label{eq:kappa0infty}
  \kappa\bigl(K_{Y_\lambda}+(1-\varepsilon) D_\lambda \bigr) =
  \kappa\bigl(\varepsilon K_{Y_\lambda} \bigr) =
  \kappa\bigl({Y_\lambda} \bigr) = -\infty.
\end{equation}
Now choose one $\varepsilon$ and run the log minimal model program for the dlt
pair $\bigl(Y_\lambda, (1-\varepsilon)D_\lambda\bigr)$. This way one obtains
morphisms and birational maps as follows
$$
\xymatrix{
 Y_\lambda \ar@{-->}[rrr]^{\mu}_{\text{minimal model program}} &&& Y_\mu
 \ar[rrr]^{\pi}_{\text{Mori fibre space}} &&& Z.
}
$$
Again, let $D_\mu := \mu_*(D_\lambda)$ be the cycle-theoretic image of
$D_\lambda$. The main properties of $Y_\mu$ and $D_\mu$ are summarised as
follows.
\begin{enumerate}
\item The variety $Y_\mu$ is normal and $\mathbb Q$-factorial.
\item The variety $Y_\mu$ is log terminal. The pair $\bigl(Y_\mu,
  (1-\varepsilon) D_\mu \bigr)$ is dlt.
\item The divisor $K_{Y_\mu} + D_\mu$ is numerically trivial.
\item There exists a reflexive sheaf $\mathcal A_\mu \subseteq \Sym^{[m]}
  \Omega^1_{Y_\mu} (\log D_\mu)$ of rank one and Kodaira-Iitaka dimension
  $\kappa(\sA_\mu) = \dim Y_\mu$.

\end{enumerate}
In fact, more is true.

\begin{claim}\label{claim:smuklt}
  The pair $(Y_\mu, D_\mu)$ is log canonical.
\end{claim}
\begin{proof}
  Since $K_{Y_\lambda} + D_\lambda \equiv 0$, some positive multiples of
  $K_{Y_\lambda}$ and $-D_\lambda$ are numerically equivalent. For any
  two rational numbers $0 < \varepsilon', \varepsilon'' < 1$, the divisors
  $K_{Y_\lambda} + (1-\varepsilon') D_\lambda$ and $K_{Y_\lambda} +
  (1-\varepsilon'') D_\lambda$ are thus again numerically equivalent up
  to a positive rational multiple.
  
  The birational map $\mu$ is therefore a minimal model program for the pair
  $\bigl(Y_\lambda, (1-\varepsilon) D_\lambda \bigr)$, independently of the
  number $\varepsilon$ chosen in its construction.  It follows that
  $\bigl(Y_\mu, D_\mu \bigr)$ is a limit of dlt pairs and therefore log
  canonical.
\end{proof}

\subsubsection{The fibre space structure of $\boldsymbol{Y_\mu}$}
\label{ssec:fiberspace}

Another application of the ``Harder-Narasimhan-trick'' exhibits a fibre
structure of $Y_\mu$.

\begin{claim}\label{claim:55}
  The Picard-number $\rho(Y_\mu)$ is larger than one. In particular, the map
  $Y_\mu \to Z$ is a proper fibre space whose fibres are proper subvarieties
  of $Y_\mu$.
\end{claim}
\begin{proof}
  As before, let $C \subseteq Y_\mu$ be a general complete intersection
  curve. Again, the existence of the Viehweg-Zuo sheaf $\mathcal A_\mu$
  implies that the sheaf of reflexive differentials $\Omega^{[1]}_{Y_\mu}(\log
  D_\mu)$ is not semistable, and contains a destabilising subsheaf $\mathcal
  B_\mu \subseteq \Omega^{[1]}_{Y_\mu}(\log D_\mu)$ with $c_1(\mathcal
  B_\mu).C > 0$. Since the intersection number $c_1(\mathcal B_\mu).C$ is
  positive, the rank $r$ of the sheaf $\mathcal B_\mu$ must be strictly less
  than $\dim Y_\mu$, and its determinant is a subsheaf of the sheaf of
  logarithmic $r$-forms,
  $$
  \det \mathcal B_\mu \subseteq \Omega^{[r]}_{Y_\mu}(\log D_\mu)
  \quad\text{with}\quad c_1(\det \mathcal B_\mu).C > 0 \quad\text{and}\quad
  r<\dim Y_\mu.
  $$

  If $\rho(Y_\mu) = 1$, then the sheaf $\det \mathcal B_\mu$ would necessarily
  be $\mathbb Q$-ample, violating the Bogomolov-Sommese vanishing theorem for
  log canonical pairs, Corollary~\ref{cor:BS}. This finishes the proof of
  Claim~\ref{claim:55}.
\end{proof}

Now, if $F \subset Y_\mu$ is a general fibre of $\pi$ and $D_F := D_\mu \cap
F$, then $F$ is a normal curve or surface, and the pair $(F,D_F)$ is log
canonical and has Kodaira dimension $\kappa(K_F + D_F) = 0$. By
\cite[Prop.~4.11]{KM98}, the variety $F$ is even $\mathbb Q$-factorial. It is
then possible to argue by induction: assuming that Viehweg's conjecture holds
for families over surfaces, one obtains that the restriction of the family
$f^\circ$ to the strict transform $(\mu\circ\lambda)_*^{-1}(F)$ cannot be of
maximal variation. Since the fibres dominate the variety, this contradicts the
assumption that the family $f^\circ$ is of maximal variation, and therefore
finishes the sketch of proof of Proposition~\ref{prop:vzsimple2}.

\medskip

The reader interested in more details is referred to \cite[Sect.~9]{KK08c},
where a stronger statement is shown, proving that any family over a base
manifold with $\kappa(Y^\circ)=0$ is actually isotrivial.

\providecommand{\bysame}{\leavevmode\hbox to3em{\hrulefill}\thinspace}
\providecommand{\MR}{\relax\ifhmode\unskip\space\fi MR}
\providecommand{\MRhref}[2]{%
  \href{http://www.ams.org/mathscinet-getitem?mr=#1}{#2}
}
\providecommand{\href}[2]{#2}


\begin{thebibliography}{GKKP10}

\bibitem[Ara71]{Arakelov71}
{\sc S.~J. Arakelov}: \emph{Families of algebraic curves with fixed
  degeneracies}, Izv. Akad. Nauk SSSR Ser. Mat. \textbf{35} (1971), 1269--1293.
  {\sf\scriptsize MR0321933 (48 \#298)}

\bibitem[Ara10]{Ara}
{\sc C. Araujo}: \emph{The cone of pseudo-effective divisors of log varieties after
Batyrev}, Math. Z., \textbf{264} (2010), no.~1, 179--193,   \href{http://dx.doi.org/10.1007/s00209-008-0457-8}{DOI:10.1112/S0010437X09004321}.

\bibitem[Bau07]{Baum}
{\sc J.~Baum}: \emph{Aufblasungen und {D}esingularisierungen},
  Staatsexamensarbeit, Universität zu Köln,  2007.

\bibitem[BDIP02]{MR1924513}
{\sc J.~Bertin, J.-P. Demailly, L.~Illusie, and C.~Peters}: \emph{Introduction
  to {H}odge theory}, SMF/AMS Texts and Monographs, vol.~8, American
  Mathematical Society, Providence, RI, 2002, Translated from the 1996 French
  original by James Lewis and Peters. {\sf\scriptsize 1924513 (2003g:14009)}

\bibitem[Cam04]{Cam04}
{\sc F.~Campana}: \emph{Orbifolds, special varieties and classification
  theory}, Ann. Inst. Fourier (Grenoble) \textbf{54} (2004), no.~3, 499--630.
  {\sf\scriptsize MR2097416 (2006c:14013)}

\bibitem[Cam08]{Cam07}
{\sc F.~Campana}: \emph{Orbifoldes sp{\'e}ciales et classification
  bim{\'e}romorphe des vari{\'e}t{\'e}s {K}{\"a}hl{\'e}riennes compactes},
  preprint \href{http://arxiv.org/abs/0705.0737v5}{arXiv:0705.0737v5}, October
  2008.

\bibitem[{Cor}07]{MR2352762}
{\sc A.~{Corti et al.}}: \emph{Flips for 3-folds and 4-folds}, Oxford Lecture
  Series in Mathematics and its Applications, vol.~35, Oxford University Press,
  Oxford, 2007. {\sf\scriptsize MR2352762 (2008j:14031)}

\bibitem[dJS04]{deJongStarr}
{\sc A.~J. de~Jong and J.~Starr}: \emph{Cubic fourfolds and spaces of rational
  curves}, Illinois J. Math. \textbf{48} (2004), no.~2, 415--450.
  {\sf\scriptsize MR2085418 (2006e:14007)}

\bibitem[Del70]{Deligne70}
{\sc P.~Deligne}: \emph{\'{E}quations diff\'erentielles \`a points singuliers
  r\'eguliers}, Springer-Verlag, Berlin, 1970, Lecture Notes in Mathematics,
  Vol. 163. {\sf\scriptsize 54 \#5232}

\bibitem[EV82]{RevI}
{\sc H.~Esnault and E.~Viehweg}: \emph{Rev\^etements cycliques}, Algebraic
  threefolds ({V}arenna, 1981), Lecture Notes in Math., vol. 947, Springer,
  Berlin, 1982, pp.~241--250. {\sf\scriptsize MR672621 (84m:14015)}

\bibitem[EV90]{EV90}
{\sc H.~Esnault and E.~Viehweg}: \emph{Effective bounds for semipositive
  sheaves and for the height of points on curves over complex function fields},
  Compositio Math. \textbf{76} (1990), no.~1-2, 69--85, Algebraic geometry
  (Berlin, 1988). {\sf\scriptsize MR1078858 (91m:14038)}

\bibitem[EV92]{EV92}
{\sc H.~Esnault and E.~Viehweg}: \emph{Lectures on vanishing theorems}, DMV
  Seminar, vol.~20, Birkh\"auser Verlag, Basel, 1992. Available from the
  author's web site at \href{http://www.uni-due.de/~mat903/books.html}{http://www.uni-due.de/$\sim$mat903/books.html}. {\sf\scriptsize MR1193913
  (94a:14017)}

\bibitem[Fle88]{Flenner88}
{\sc H.~Flenner}: \emph{Extendability of differential forms on nonisolated
  singularities}, Invent. Math. \textbf{94} (1988), no.~2, 317--326.
  {\sf\scriptsize MR958835 (89j:14001)}

\bibitem[GR70]{GR70}
{\sc H.~Grauert and O.~Riemenschneider}: \emph{Verschwindungss\"atze f\"ur
  analytische {K}ohomologiegruppen auf komplexen {R}\"aumen}, Invent. Math.
  \textbf{11} (1970), 263--292. {\sf\scriptsize MR0302938 (46 \#2081)}

\bibitem[GKK10]{GKK08}
{\sc D.~Greb, S.~Kebekus, and S.J.~Kov\'acs}: \emph{Extension theorems for
  differential forms, and {B}ogomolov-{S}ommese vanishing on log canonical
  varieties}, Compos.~Math. \textbf{146} (2010), 193--219,
  \href{http://dx.doi.org/10.1112/S0010437X09004321}{DOI:10.1112/S0010437X0900%
4321}.

\bibitem[GKKP11]{GKKP10} 
  {\sc D.~Greb, S.~Kebekus, S.J.~Kovács, and T.~Peternell}: \emph{Differential
    forms on log canonical spaces}, Publ. Math. IHES. 114 (2011), no. 1,
  87–169. \href{http://dx.doi.org/10.1007/s10240-011-0036-0}{DOI:10.1007/s10240-011-0036-0}. An
  extended version with additional graphics is available as
  \href{http://arxiv.org/abs/1003.2913}{arXiv:1003.2913}.

\bibitem[HK10]{HK10}
{\sc C.~D. Hacon and S.J.~Kovács}: \emph{Classification of higher dimensional
  algebraic varieties}, Birkhäuser, May 2010.

\bibitem[HM07]{HMcK07}
{\sc C.~D. Hacon and J.~McKernan}: \emph{On {S}hokurov's rational connectedness
  conjecture}, Duke Math. J. \textbf{138} (2007), no.~1, 119--136.
  {\sf\scriptsize MR2309156 (2008f:14030)}

\bibitem[Har70]{Hartshorne1970}
{\sc R.~Hartshorne}: \emph{Ample subvarieties of algebraic varieties}, Notes
  written in collaboration with C. Musili. Lecture Notes in Mathematics, Vol.
  156, Springer-Verlag, Berlin, 1970. {\sf\scriptsize MR0282977 (44 \#211)}

\bibitem[Har77]{Ha77}
{\sc R.~Hartshorne}: \emph{Algebraic geometry}, Springer-Verlag, New York,
  1977, Graduate Texts in Mathematics, No. 52. {\sf\scriptsize MR0463157 (57
  \#3116)}

\bibitem[Huy05]{Huy05}
{\sc D.~Huybrechts}: \emph{Complex geometry}, Universitext, Springer-Verlag,
  Berlin, 2005, An introduction. {\sf\scriptsize MR2093043 (2005h:32052)}

\bibitem[HL97]{HL97}
{\sc D.~Huybrechts and M.~Lehn}: \emph{The geometry of moduli spaces of
  sheaves}, Aspects of Mathematics, E31, Friedr. Vieweg \& Sohn, Braunschweig,
  1997. {\sf\scriptsize MR1450870 (98g:14012)}

\bibitem[Iit82]{Iitaka82}
{\sc S.~Iitaka}: \emph{Algebraic geometry}, Graduate Texts in Mathematics,
  vol.~76, Springer-Verlag, New York, 1982, An introduction to birational
  geometry of algebraic varieties, North-Holland Mathematical Library, 24.
  {\sf\scriptsize MR637060 (84j:14001)}

\bibitem[JK09a]{JKSpecialBaseManifolds} {\sc K.~Jabbusch and S.~Kebekus}:
  \emph{Families over special base manifolds and a conjecture of {C}ampana},
  Math. Z., \textbf{269} (2011), no.~3, 847--878,
  \href{http://dx.doi.org/10.1007/s00209-010-0758-6}{DOI:10.1007/s00209-010-0758-6}.

\bibitem[JK09b]{JK09}
  {\sc K.~Jabbusch and S.~Kebekus}: \emph{Positive sheaves of differentials on
    coarse moduli spaces}, preprint
  \href{http://arxiv.org/abs/0904.2445}{arXiv:0904.2445}, April 2009. To
  appear in Annales de l'Institut Fourier, Vol. 61 (2011).

\bibitem[KK07]{KK07b}
{\sc S.~Kebekus and S.~J. Kov{\'a}cs}: \emph{The structure of surfaces mapping
  to the moduli stack of canonically polarized varieties}, preprint
  \href{http://arxiv.org/abs/0707.2054}{arXiv:0707.2054}, July 2007.

\bibitem[KK08a]{KK08}
{\sc S.~Kebekus and S.~J. Kov{\'a}cs}: \emph{Families of canonically polarized
  varieties over surfaces}, Invent. Math. \textbf{172} (2008), no.~3, 657--682,
  \href{http://dx.doi.org/10.1007/s00222-008-0128-8}{DOI:10.1007/s00222-008-01%
28-8}. {\sf\scriptsize MR2393082}

\bibitem[KK08b]{KK08b}
{\sc S.~Kebekus and S.~J. Kov{\'a}cs}: \emph{Families of varieties of general
  type over compact bases}, Adv. Math. \textbf{218} (2008), no.~3, 649--652,
  \href{http://dx.doi.org/10.1016/j.aim.2008.01.005}{DOI:10.1016/j.aim.2008.01%
.005}. {\sf\scriptsize MR2414316 (2009d:14042)}

\bibitem[KK08c]{KK08c}
{\sc S.~Kebekus and S.~J. Kov{\'a}cs}: \emph{The structure of surfaces and
    threefolds mapping to the moduli stack of canonically polarized
    varieties}, Duke Math. J. \textbf{155} (2010), no.~1, 1--33.

\bibitem[KS06]{KS06}
{\sc S.~Kebekus and L.~{Sol{\'a} Conde}}: \emph{Existence of rational curves on
  algebraic varieties, minimal rational tangents, and applications}, Global
  aspects of complex geometry, Springer, Berlin, 2006, pp.~359--416.
  {\sf\scriptsize MR2264116}

\bibitem[KST07]{KST07}
{\sc S.~Kebekus, L.~{Sol\'a Conde}, and M.~Toma}: \emph{Rationally connected
  foliations after {Bogomolov} and {McQuillan}}, J. Algebraic Geom. \textbf{16}
  (2007), no.~1, 65--81.

\bibitem[Kol86]{Kol86}
{\sc J.~Koll{\'a}r}: \emph{Higher direct images of dualizing sheaves. {I}},
  Ann. of Math. (2) \textbf{123} (1986), no.~1, 11--42. {\sf\scriptsize
  MR825838 (87c:14038)}

\bibitem[Kol96]{K96}
{\sc J.~Koll{\'a}r}: \emph{Rational curves on algebraic varieties}, Ergebnisse
  der Mathematik und ihrer Grenzgebiete. 3. Folge. A Series of Modern Surveys
  in Mathematics, vol.~32, Springer-Verlag, Berlin, 1996. {\sf\scriptsize
  MR1440180 (98c:14001)}

\bibitem[KK10]{KKLogCanonicalDuBois}
  {\sc J.~Koll{\'a}r and S.~J. Kov{\'a}cs}: \emph{Log canonical singularities
    are {D}u~{B}ois}, J. Amer. Math. Soc. \textbf{23} (2010), 791--813.
  \href{http://dx.doi.org/10.1090/S0894-0347-10-00663-6}{DOI:10.1090/S0894-034%
    7-10-00663-6}

\bibitem[{\hbox{K}}{\hbox{M}}98]{KM98}
{\sc J.~{\hbox{K}}oll{\'a}r and S.~{\hbox{M}}ori}: \emph{Birational geometry of
  algebraic varieties}, Cambridge Tracts in Mathematics, vol. 134, Cambridge
  University Press, Cambridge, 1998, With the collaboration of C. H. Clemens
  and A. Corti, Translated from the 1998 Japanese original. {\sf\scriptsize
  2000b:14018}

\bibitem[{Kol}92]{FandA92}
{\sc J.~{Koll{\'a}r et al.}}: \emph{Flips and abundance for algebraic
  threefolds}, Soci\'et\'e Math\'ematique de France, Paris, 1992, Papers from
  the Second Summer Seminar on Algebraic Geometry held at the University of
  Utah, Salt Lake City, Utah, August 1991, Ast{\'e}risque No. 211 (1992).
  {\sf\scriptsize MR1225842 (94f:14013)}

\bibitem[Kov96]{Kovacs96e}
{\sc S.~J. Kov{\'a}cs}: \emph{Smooth families over rational and elliptic
  curves}, J. Algebraic Geom. \textbf{5} (1996), no.~2, 369--385, Erratum:
  {J}.\ {A}lgebraic {G}eom.\ {\bf 6} (1997), no.\ 2, 391. {\sf\scriptsize
  MR1374712 (97c:14035)}

\bibitem[Kov97]{Kovacs97a}
{\sc S.~J. Kov{\'a}cs}: \emph{On the minimal number of singular f{i}bres in a
  family of surfaces of general type}, J. Reine Angew. Math. \textbf{487}
  (1997), 171--177. {\sf\scriptsize MR1454264 (98h:14038)}

\bibitem[Kov00]{Kovacs00a}
{\sc S.~J. Kov{\'a}cs}: \emph{Algebraic hyperbolicity of f{i}ne moduli spaces},
  J. Algebraic Geom. \textbf{9} (2000), no.~1, 165--174. {\sf\scriptsize
  MR1713524 (2000i:14017)}

\bibitem[Kov06]{Kovacs06a}
{\sc S.~J. Kov{\'a}cs}: \emph{Subvarieties of moduli stacks of canonically
  polarized varieties}, Proceedings of the AMS Summer Research Institute held
  at the University of Washington, Seattle, WA, July 25--August 12, 2005,
  Proceedings of Symposia in Pure Mathematics, American Mathematical Society,
  Providence, RI, 2006, to appear.

\bibitem[Kov09]{MR2483953}
{\sc S.~J. Kov{\'a}cs}: \emph{Young person's guide to moduli of higher
  dimensional varieties}, Algebraic geometry---{S}eattle 2005. {P}art 2, Proc.
  Sympos. Pure Math., vol.~80, Amer. Math. Soc., Providence, RI, 2009,
  pp.~711--743. {\sf\scriptsize MR2483953 (2009m:14051)}

\bibitem[KS09]{Kovacs-Schwede09}
{\sc S.~J. Kov{\'a}cs and K.~Schwede}: \emph{Hodge theory meets the minimal
    model program: a survey of log canonical and {{D}u~{B}ois} singularities},
  preprint, 2009.  \href{http://arxiv.org/abs/0909.0993v1}{arXiv:0909.0993v1}

\bibitem[Laz04]{L04}
{\sc R.~Lazarsfeld}: \emph{Positivity in algebraic geometry. {II}}, Ergebnisse
  der Mathematik und ihrer Grenzgebiete. 3. Folge. A Series of Modern Surveys
  in Mathematics [Results in Mathematics and Related Areas. 3rd Series. A
  Series of Modern Surveys in Mathematics], vol.~49, Springer-Verlag, Berlin,
  2004, Positivity for vector bundles, and multiplier ideals. {\sf\scriptsize
  MR2095472}

\bibitem[Liu02]{Liu02}
{\sc Q.~Liu}: \emph{Algebraic geometry and arithmetic curves}, Oxford Graduate
  Texts in Mathematics, vol.~6, Oxford University Press, Oxford, 2002,
  Translated from the French by Reinie Ern{\'e}, Oxford Science Publications.
  {\sf\scriptsize MR1917232 (2003g:14001)}

\bibitem[Loh11]{Loh11}
{\sc D.~Lohmann}: \emph{Families of canonically polarized manifolds over log Fano varieties},
  preprint \href{http://arxiv.org/abs/1107.4545}{arXiv:1107.4545}, 
  July 2011.

\bibitem[Mat02]{Matsuki02}
{\sc K.~Matsuki}: \emph{Introduction to the {M}ori program}, Universitext,
  Springer-Verlag, New York, 2002. {\sf\scriptsize 2002m:14011}

\bibitem[Mig95]{Migliorini95}
{\sc L.~Migliorini}: \emph{A smooth family of minimal surfaces of general type
  over a curve of genus at most one is trivial}, J. Algebraic Geom. \textbf{4}
  (1995), no.~2, 353--361. {\sf\scriptsize MR1311355 (95m:14023)}

\bibitem[Miy87]{Miy85}
{\sc Y.~Miyaoka}: \emph{Deformations of a morphism along a foliation and
  applications}, Algebraic geometry, Bowdoin, 1985 (Brunswick, Maine, 1985),
  Proc. Sympos. Pure Math., vol.~46, Amer. Math. Soc., Providence, RI, 1987,
  pp.~245--268. {\sf\scriptsize MR927960 (89e:14011)}

\bibitem[Nam01]{Namikawa01}
{\sc Y.~Namikawa}: \emph{Extension of 2-forms and symplectic varieties}, J.
  Reine Angew. Math. \textbf{539} (2001), 123--147. {\sf\scriptsize MR1863856
  (2002i:32011)}

\bibitem[OV01]{Oguiso-Viehweg01}
{\sc K.~Oguiso and E.~Viehweg}: \emph{On the isotriviality of families of
  elliptic surfaces}, J. Algebraic Geom. \textbf{10} (2001), no.~3, 569--598.
  {\sf\scriptsize MR1832333 (2002d:14054)}

\bibitem[OSS80]{OSS}
{\sc C.~Okonek, M.~Schneider, and H.~Spindler}: \emph{Vector bundles on complex
  projective spaces}, Progress in Mathematics, vol.~3, Birkh\"auser Boston,
  Mass., 1980. {\sf\scriptsize MR561910 (81b:14001)}

\bibitem[Par68]{Parshin68}
{\sc A.~N. Parshin}: \emph{Algebraic curves over function fields. {I}}, Izv.
  Akad. Nauk SSSR Ser. Mat. \textbf{32} (1968), 1191--1219. {\sf\scriptsize
  MR0257086 (41 \#1740)}

\bibitem[Pat11]{Pat11}
{\sc Z.~Patakfalvi}: \emph{Viehweg's hyperbolicity conjecture is true over
    compact bases}, preprint \href{http://arxiv.org/abs/1109.2835}{arXiv:1109.2835},
  September 2011.

\bibitem[Rei87]{Reid87}
{\sc M.~Reid}: \emph{Young person's guide to canonical singularities},
  Algebraic geometry, Bowdoin, 1985 (Brunswick, Maine, 1985), Proc. Sympos.
  Pure Math., vol.~46, Amer. Math. Soc., Providence, RI, 1987, pp.~345--414.
  {\sf\scriptsize MR927963 (89b:14016)}

\bibitem[Sha63]{Shaf63}
{\sc I.~R. Shafarevich}: \emph{Algebraic number fields}, Proc. Internat. Congr.
  Mathematicians (Stockholm, 1962), Inst. Mittag-Leffler, Djursholm, 1963,
  English translation: Amer.\ Math.\ Soc.\ Transl.\ (2) {\bf 31} (1963),
  25--39, pp.~163--176. {\sf\scriptsize MR0202709 (34 \#2569)}

\bibitem[SvS85]{SS85}
{\sc J.~Steenbrink and D.~van Straten}: \emph{Extendability of holomorphic
  differential forms near isolated hypersurface singularities}, Abh. Math. Sem.
  Univ. Hamburg \textbf{55} (1985), 97--110. {\sf\scriptsize MR831521
  (87j:32025)}

\bibitem[Ste85]{Steenbrink85}
{\sc J.~H.~M. Steenbrink}: \emph{Vanishing theorems on singular spaces},
  Ast\'erisque (1985), no.~130, 330--341, Differential systems and
  singularities (Luminy, 1983). {\sf\scriptsize MR804061 (87j:14026)}

\bibitem[Vie83]{Viehweg83}
{\sc E.~Viehweg}: \emph{Weak positivity and the additivity of the {K}odaira
  dimension for certain fibre spaces}, Algebraic varieties and analytic
  varieties (Tokyo, 1981), Adv. Stud. Pure Math., vol.~1, North-Holland,
  Amsterdam, 1983, pp.~329--353. {\sf\scriptsize MR715656 (85b:14041)}

\bibitem[Vie95]{V95}
{\sc E.~Viehweg}: \emph{Quasi-projective moduli for polarized manifolds},
  Ergebnisse der Mathematik und ihrer Grenzgebiete (3) [Results in Mathematics
  and Related Areas (3)], vol.~30, Springer-Verlag, Berlin, 1995.
  {\sf\scriptsize MR1368632 (97j:14001)}

\bibitem[Vie01]{Viehweg01}
{\sc E.~Viehweg}: \emph{Positivity of direct image sheaves and applications to
  families of higher dimensional manifolds}, School on Vanishing Theorems and
  Effective Results in Algebraic Geometry (Trieste, 2000), ICTP Lect. Notes,
  vol.~6, Abdus Salam Int. Cent. Theoret. Phys., Trieste, 2001, Available on
  the \href{http://www.ictp.trieste.it/~pub_off/services}{ICTP web site},
  pp.~249--284. {\sf\scriptsize MR1919460 (2003f:14024)}

\bibitem[VZ02]{VZ02}
{\sc E.~Viehweg and K.~Zuo}: \emph{Base spaces of non-isotrivial families of
  smooth minimal models}, Complex geometry (G\"ottingen, 2000), Springer,
  Berlin, 2002, pp.~279--328. {\sf\scriptsize MR1922109 (2003h:14019)}

\bibitem[VZ03]{Vie-Zuo03a}
{\sc E.~Viehweg and K.~Zuo}: \emph{On the {B}rody hyperbolicity of moduli
  spaces for canonically polarized manifolds}, Duke Math. J. \textbf{118}
  (2003), no.~1, 103--150. {\sf\scriptsize MR1978884 (2004h:14042)}

\bibitem[Voi07]{Voisin-Hodge1}
{\sc C.~Voisin}: \emph{Hodge theory and complex algebraic geometry. {I}},
  english ed., Cambridge Studies in Advanced Mathematics, vol.~76, Cambridge
  University Press, Cambridge, 2007, Translated from the French by Leila
  Schneps. {\sf\scriptsize MR2451566 (2009j:32014)}

\bibitem[Zuo00]{Zuo00}
{\sc K.~Zuo}: \emph{On the negativity of kernels of {K}odaira-{S}pencer maps on
  {H}odge bundles and applications}, Asian J. Math. \textbf{4} (2000), no.~1,
  279--301, Kodaira's issue. {\sf\scriptsize MR1803724 (2002a:32011)}

\end{thebibliography}
\end{document}